\documentclass[a4paper, 11pt]{amsart}

\usepackage{amsmath, amssymb, amsthm}%, wasysym}
\usepackage[english]{babel}
\usepackage[T1]{fontenc}
\newcommand{\ee}{\mathbf{e}}

\def\e{\epsilon}

\def\eps {\epsilon}

\def \be {  \varpi}

\newcommand{\fer}[1]{(\ref{#1})}
\newcommand{\R}{\mathbb R}
\def\be#1\ee{\begin{equation}#1\end{equation}}
\def\var{\varepsilon}

\newcommand{\bq}{\begin{equation}}
\newcommand{\eq}{\end{equation}}

\newtheorem{thm}{Theorem}

\theoremstyle{remark}
\newtheorem{rem}{Remark}
\theoremstyle{definition}

\newenvironment{equations}{\equation\aligned}{\endaligned\endequation}
\begin{document}

\title[Non-Maxwellian kinetic equations modeling  wealth distribution]{NON-MAXWELLIAN KINETIC EQUATIONS MODELING THE EVOLUTION OF WEALTH DISTRIBUTION}

\author{GIULIA FURIOLI}
\address{DIGIP, University of Bergamo, viale Marconi 5, 24044 Dalmine, Italy}
\email{giulia.furioli@unibg.it} 

\author{ADA PULVIRENTI}
\address{Department of Mathematics, University of Pavia, 
via Ferrata 1,
Pavia, 27100 Italy}
\email{ada.pulvirenti@unipv.it}

\author{ELIDE TERRANEO}
\address{Department of Mathematics,
University of Milan, via Saldini 50, 20133 Milano, Italy }
\email{elide.terraneo@unimi.it}

\author{GIUSEPPE TOSCANI}
\address{Department of Mathematics, University of Pavia,
via Ferrata 1,
Pavia, 27100 Italy}
\email{giuseppe.toscani@unipv.it}

\maketitle

\begin{abstract}
We introduce a class of new one-dimensional linear Fokker--Planck type equations describing the evolution in time of the wealth in a multi-agent society. The equations are obtained, via a standard limiting procedure, by introducing an economically relevant variant to the kinetic model introduced in 2005 by Cordier, Pareschi and Toscani according to previous studies by Bouchaud and M\'ezard. The steady state of wealth predicted by these new Fokker--Planck equations remains unchanged with respect to the steady state of the original Fokker--Planck equation. However, unlike the original equation, it is proven by a new logarithmic Sobolev inequality with weight and classical entropy methods that the solution converges exponentially fast to equilibrium.    
\end{abstract}
\vskip 3mm


\section{Introduction}

In recent years, mathematical modeling of multi-agent systems became a challenging and productive field of research involving both applied mathematicians and physicists. Among other aspects, this research activity introduced new applications of statistical physics to interdisciplinary fields ranging from  the classical biological context \cite{BT15,BDK,GR,Ha1,KP12,RT,Tos13}, to the new aspects of socio-economic dynamics \cite{CCCC,CFL,NPT,PT13,SC}.

A significant part of this activity was devoted to the study of classical problems in economy \cite{BaTo,BaTo2,BST,Ch02,ChaCha00,CCM,ChChSt05,CPP,DY00,DMT,DMT1,gup,IKR98,MD,MaTo07,Sl04,TBD}, including  the important aspect of the justification of Pareto's discovery of fat tails in wealth distribution of western societies \cite{Par}
Beside the kinetic models of Boltzmann type introduced in recent years to enlighten the formation of an unequal distribution of wealth among trading agents \cite{CCCC,PT13}, a Fokker--Planck type equation assumed a leading role. This equation, which describes the time-evolution of the density $f(t,w)$ of a system of agents with personal wealth $w\ge0$ at time $t \ge 0$ reads
  \be\label{FP2c}
 \frac{\partial f(t,w)}{\partial t} = J(f)(t,w) = \frac \sigma{2}\frac{\partial^2 }
 {\partial w^2}\left( w^2 f(t,w)\right) + \lambda \frac{\partial }{\partial v}\left(
 (w-m) f(t,w)\right).
 \ee
In \fer{FP2c},  $\lambda$, $\sigma$ and $m$ denote  positive constants related to essential properties of the trade rules of the agents.  

The Fokker--Planck equation \fer{FP2c} has been first obtained by Bouchaud and M\'ezard in Ref. \cite{BM} through a mean field limit procedure applied to a stochastic dynamical equation for the wealth density.  The same equation was subsequently found in Ref. \cite{CoPaTo05} by resorting to an asymptotic procedure applied to a Boltzmann-type kinetic model for binary trading in presence of risks. 
 
 The importance of equation \fer{FP2c} in the study of wealth distribution is related to the economic relevance of its stationary solution density, given by the inverse
Gamma density  \cite{BM,CoPaTo05}
 \be\label{equi2}
f_\infty(w) =\frac{(\mu\, m)^{1+\mu}}{\Gamma(1+\mu)}\frac{\exp\left(-\frac{\mu\, m}{w}\right)}{w^{2+\mu}}.
 \ee
  In \fer{equi2} $\mu$ denotes the positive constant
  \be\label{mu}
  \mu =2 \frac{\lambda}{\sigma}.
  \ee
This stationary distribution,  as predicted by the  analysis of the italian economist Vilfredo Pareto \cite{Par}, exhibits a power-law tail for large values of the wealth  variable.

In addition to the references \cite{BM} and  \cite{CoPaTo05}, the same equation with a modified drift term appears when considering suitable asymptotics of Boltzmann-type equations for binary trading in presence of taxation \cite{Bisi}, in the case in which taxation is described by the redistribution operator introduced in Ref. \cite{BST}. Systems of Fokker--Planck equations of type \fer{FP2c} have been considered in Ref. \cite{DT08} to model wealth distribution in different countries which are coupled by mixed trading. Further, the operator $J(f)$ in equation \fer{FP2c} and its equilibrium kernel density have been considered in a non homogeneous setting to obtain Euler-type equations describing the joint evolution of wealth and propensity to invest \cite{DT07},  to study the evolution of wealth in a society with agents using personal knowledge to trade \cite{PT14}, and to observe the consequences on the distribution of wealth by exercising a control at the level of microscopic interactions \cite{DPT}.
Last, equation \fer{FP2c} has been studied with data in the whole real line (thus allowing agents to have debts) in Ref. \cite{TT18}, by showing that the large-time behavior of the solution remains unchanged if the quantity of debts does not destroy the positivity of the mean wealth.

These results contributed to retain that this equation represents a quite satisfactory description of the time-evolution of wealth density towards a Pareto-type equilibrium in a trading society. 

However, as far as the large-time behavior is concerned, convergence to equilibrium results for the solution to \fer{FP2c} are not fully satisfactory, since convergence in strong sense was proven to hold or at a polynomial rate for general initial densities \cite{TT}, or at exponential rate for a very restricted class of initial densities, that need to be chosen \emph{very close} to the equilibrium density \cite{FPTT17} This makes an essential difference between equation \fer{FP2c} and the standard Fokker--Planck equation, where convergence to equilibrium has been proven to hold exponentially in time with an explicit rate for all initial data satisfying natural and non restrictive physical conditions \cite{Tos99} The physical way to study convergence is to resort to the decay of relative entropy, and to logarithmic Sobolev inequalities \cite{AMTU,OV} Unlikely,  as discussed in Ref. \cite{MJT}, and more recently in Ref. \cite{FPTT17} (cf. also Ref. \cite{TT18}), this type of inequalities do not seem to be available in presence of variable diffusion coefficients as they are those in \fer{FP2c}.    

As a matter of fact, the proof of exponential convergence to equilibrium is not only a pure mathematical result. The (apparent) lack of exponential convergence to equilibrium for the solution to equation \fer{FP2c} brings into question if the mathematical modeling leading to the Fokker--Planck equation \fer{FP2c} takes into account in a right way all the principal aspects of the agents trading. Indeed, like it happens for the solution to the famous Boltzmann equation \cite{TV,Vi1}, exponential convergence to equilibrium is the main motivation to justify why in the real world we are always observing a wealth distribution with Pareto tails. 

In this paper we identify one of the modeling points that should be improved in the choice made at a kinetic level in Ref. \cite{CoPaTo05}, where the Boltzmann collision operator has been selected to be of Maxwell type \cite{Bob,Cer}. In classical kinetic theory, Maxwellian pseudo-molecules describe a gas in which the collision kernel does not depend on the relative velocity of the molecules. Analogously, in the economic context, the Maxwellian hypothesis corresponds to make the strong assumption that the trading between agents does not depend on the amount of wealth put into the trade. While this choice is easy to handle with from a mathematical point of view, it naturally leads to eliminate a part of human behavior from the trading. Consequently, this simplification could introduce some fault into the model, since, in contrast to classical kinetic theory of rarefied gases, various aspects of human behavior play a substantial role into the mathematical modeling of socio-economic phenomena \cite{ABG16,BCKS,BHT,BKS,Maini08,BS,GT-ec}.

With the aim of improving the model, we will introduce in the underlying kinetic equation of Boltzmann type a variable collision kernel which is designed to exclude unphysical interactions between agents. Within this choice, we will obtain in the limit a new class of Fokker--Planck equations which contain equation \fer{FP2c} as a particular case. The new Fokker--Planck equations read
\begin{equation}\label{FP}
 \frac{\partial f(t,w)}{\partial t} =  J_\delta(f)(t,w) = \frac\sigma 2 \frac{\partial^2 }{\partial w^2}
 \left(w^{2+\delta}f(t,w)\right )+ 
\lambda\, \frac{\partial}{\partial w}\left( w^\delta(w - m ) f(t,w)\right).
 \end{equation}
 In \fer{FP}  $\delta$ is a positive constant, with $0<\delta\leq1$. Equation \fer{FP} has a unique equilibrium density of unit mass, given by the inverse Gamma function
 \be\label{new-equ}
 f_\infty^\delta(w) = \frac{(\mu\, m)^{1+\delta + \mu}}{\Gamma(1+\delta + \mu)}\frac{\exp\left(-\frac{\mu\, m}{w}\right)}{w^{2+\delta +\mu}}.
 \ee
In \fer{new-equ} $\mu$ is the positive constant defined in \fer{mu}. Hence, the presence of the constant $\delta$ is such that the Pareto index in the equilibrium density of the target Fokker--Planck equation is increased by the amount $\delta$. 

The main novelty related to the Fokker--Planck equation \fer{FP} we are going to prove (Theorem \ref{expon}) is that its solution can be shown to converge exponentially, at an explicit rate, towards the equilibrium  density \fer{new-equ}. For this reason, we believe that equation \fer{FP} furnishes a better description of the process of relaxation of the wealth distribution density in a multi-agent society.

To end this introduction, we outline again the importance of taking into account typical aspects of human behavior in the mathematical modeling of multi-agent systems. Following this line of thought, a recent result on service times in a call-center \cite{GT17}, further generalized to various skewed phenomena in Ref. \cite{GT18}, led to build, on the basis of the well consolidated prospect theory of Kahneman and Twersky \cite{KT,KT1}, a Fokker--Planck equation with a lognormal density as a steady state.  While this equation is consistent with the huge amount of data available for the phenomena under study, it has the further feature that the solution density is exponentially convergent towards equilibrium. 

It is noticeable that  various and essential aspects related to human behavior of agent trading have been fully considered in the binary interactions considered in Ref. \cite{CoPaTo05} beyond  the choice of a constant collision kernel in the Boltzmann collision operator. Indeed, in addition to risk, which is naturally part of the trading process, one of the fundamental assumptions there was to take into account, according to the proposal of Chakrabarti and coworkers \cite{Ch02,ChaCha00,CCM}, the agent's tendency to save  using only a small part of money in a single trade, the so-called \emph{saving propensity}. 
 
 Let us describe now the structure of the paper. In Section \ref{model2} we will illustrate the reasons behind the modification made at the model considered in Ref. \cite{CoPaTo05}, which lead to introduce a variable collision kernel. Then, in Section \ref{gambling} we show how this choice modifies the explicit stationary solution of the simple kinetic model known as \emph{pure gambling model}.  Section \ref{quasi} deals with the asymptotic procedure leading from the kinetic description of Boltzmann type to the Fokker--Planck equation \fer{FP} in presence of a variable collision kernel and a linear interaction. Finally, in Section \ref{large-time} we study the large-time behavior of the solution to the Fokker--Planck equation, by showing exponential convergence towards the equilibrium density \fer{new-equ} in relative entropy. The result is achieved owing to  a new  logarithmic Sobolev type inequality (Theorem \ref{LS-theo}) which relates Fisher information with weight with Shannon relative entropy.  This inequality is obtained along the lines of the proof of the pioneering paper by Otto and Villani \cite{OV}. In addition, we will show that exponential convergence in $L^1$-norm towards equilibrium holds for a larger class of solutions departing from initial data satisfying suitable conditions at the boundary.

Last, Section \ref{n-l} will deal with purely nonlinear models, directly derived from the bilinear kinetic equation of Boltzmann type introduced in Ref. \cite{CoPaTo05}, in presence of the non-Maxwellian collision kernel considered in Section \ref{quasi}. The main achievement is that the presence of the kernel is such that the limit Fokker--Planck equation, while preserving the mean wealth, has the coefficients of diffusion and drift that depend from suitable moments of the solution itself, thus introducing nonlinear effects into the evolution. 
  
%%%%%%%%%%%%%%%%%%%%%%%%%%%%%%
  \section{Kinetic modeling of trading activity}\label{model2}   

The description of  the evolution of wealth in a multi-agent system has its  roots in statistical physics, and,  in particular, on methods borrowed from kinetic theory of rarefied gases. These methods  have been successfully applied to construct master equations of Boltzmann type, usually referred to as kinetic equations, describing the
time-evolution of some characteristic of the agents, like wealth, opinion, knowledge and others, and, in some cases, to recover the emergence of universal behaviors through their equilibria \cite{CaceresToscani2007,CCCC,NPT,PT13,SC}.

The building block of kinetic modeling is represented by microscopic interactions, which, similarly to binary interactions between  velocities in the classical kinetic theory of rarefied gases, are specialized to describe the variation law  of some selected agent trait, like  wealth or opinion. Then, from the microscopic law of variation of the number density consequent to the (fixed-in-time) way of interaction, one will construct a kinetic equation able to capture both the time evolution and the steady profile of the phenomenon under study \cite{NPT,PT13}.

The population of agents (the traders)  is considered homogeneous with respect to the personal wealth.   
In addition, it is assumed that agents are indistinguishable \cite{PT13} This means that an agent's state at any instant of time $t\ge 0$ is completely characterized by the amount $w \ge 0$ of their wealth. 
Consequently, the state of the agents system will be fully characterized by the unknown density (or distribution function) $f = f(t,w)$, of the wealth $w\in \R_+$ and the time $t\ge 0$. The  time evolution of the density is described, as shown later on, by a kinetic equation of Boltzmann type.
The precise meaning of the density $f$ is the following. Given the system of traders, and given an interval or a more complex sub-domain $D  \subseteq \R_+$, the integral
\[
\int_D f(t,w)\, dw
\]
represents the number of traders which  are characterized by an amount of wealth $w \in D$  at time $t \ge 0$. It is assumed that the density function is normalized to one, that is for all $t\ge 0$
\[
\int_{\R_+} f(t,w)\, dw = 1.
\]
The change in time of the density is due to the fact that agents of the system are subject to trades, and continuously upgrade their amounts of wealth $w$  at each trade.  To maintain the connection with classical kinetic theory of rarefied gases, we will always refer to a single upgrade of the quantity $w$ as an \emph{interaction}. 

In what follows, to avoid inessential complications, and to outline the main novelties of the new approach, we will refer to a linear interaction, which takes into account all the trading aspects of the original nonlinear model considered in Ref. \cite{CoPaTo05}. According to the binary trade introduced in Ref. \cite{CoPaTo05} we  assume that the elementary change of wealth $w \in \R_+$  of an agent of the system trading with the market is the result of three different contributes
 \be\label{lin}
 w^* = (1-\lambda)w +  \lambda v  + \eta\, w.
 \ee  
In \fer{lin} $\lambda$ is a positive constant, such that $\lambda \ll 1$. The first term in \fer{lin} measures the wealth  that remains in the hands of the trader  who entered into the trading market with a (small) percentage  $\lambda w$ of his wealth. The constant $\lambda$ quantifies the saving propensity of the agent, namely the human perception that it results quite dangerous to trade the whole amount of wealth in a single interaction. The second term represents the amount of wealth the trader receives from the market as result of the trading activity.  Here $v\in \R_+$ is sampled by a certain distribution $\mathcal E$ which describes the situation of the market itself. Note that in principle the constant in front of the wealth $v$ could be different from $\lambda$, say $\chi$. However, as we shall see later on, the  choice $\lambda \not=\chi$ will not introduce essential differences in the behavior of the target model. Finally, the last term takes into account the risks of the market. In \fer{lin}  $\eta$ is a centered random variable with finite variance $\sigma \ll1$, which in general  is assumed such that $ \eta \ge -1 + \lambda$, to ensure that even in a risky trading market, the post trading wealth remains non negative.  We will further assume that the random variable $\eta$ takes values on a bounded set, that is $-1 + \lambda \le \eta \le \lambda^* < +\infty$. This condition is coherent with the trade modeling, and corresponds to put a bound from above at the possible random gain that a trader can have in a single interaction.

By classical methods of kinetic theory \cite{PT13}, and resorting to the derivation of the classical linear Boltzmann equation of elastic rarefied gases \cite{Cer}, one can easily show that  the time variation of the wealth density is due to the balance between the post and pre-interaction variations of wealth density  due to the microscopic trades of type \fer{lin}. Hence, the wealth density $f(t,w)$ of the agents system 
obeys, for all smooth functions $\varphi(v)$ (the observable quantities), to the integro-differential equation 
 \begin{equation}
  \label{line-w}
 \frac{d}{dt}\int_{\R_+}f(t,w)\varphi(w)\,dw  = 
  \Big \langle \int_{\R_+\times \R_+} K(v,w) \left( \varphi(w^*)-\varphi(w) \right) f(t,w)\mathcal E(v)
\,dv\,dw\Big \rangle.
 \end{equation}
In \fer{line-w}, the function $\mathcal E(v)$, $v\in \R_+$ is the distribution density of wealth of the market, while the function $K(v,w)$ denotes the collision kernel, which assigns to the interaction $(v,w)$ a certain probability to occur. In \fer{line-w} the notation $\langle \cdot \rangle$ denotes mathematical expectation, and takes into account the presence of the random variable $\eta$ in \fer{lin}.

In kinetic theory of rarefied gases, where the pair $v,w$ denotes the velocities of the particles which enter a collision, the collisional kernel is assumed to be function of the  relative velocity $|v-w|$, as well as, in higher dimensions, of the deflection angle \cite{Cer}.  The most important and used kernels describe hard (respectively  weak) interactions and depend on a positive (respectively negative) power $|v-w|^\delta$ of the relative velocity. A great simplification then occurs when considering Maxwellian pseudo-molecules, characterized by the value $\delta = 0$. The main consequences of this choice can be fully understood by looking at the exhaustive review by Bobylev \cite{Bob},  who first discovered the possibility to  resort to Fourier transform analysis in the nonlinear Boltzmann equation for Maxwellian pseudo-molecules. The choice of a kernel independent of the relative velocity is also at the basis of the famous one-dimensional kinetic model known as Kac caricature of the Boltzmann equation, introduced by Kac in Ref. \cite{Kac}.

In the socio-economic context, the simplification of the Maxwell molecules, leading to a constant interaction kernel, has been regularly assumed \cite{FPTT17,PT13}. This simplification, maybe not so well justified from a modeling point of view, led to the possibility, as in the Boltzmann setting, to make use of the Fourier transformed version of the underlying kinetic equations \cite{BaTo,BaTo2,DMT,DMT1,MaTo07}, and to extract from the Fourier version a number of properties about the large-time behavior of the solution to equation \fer{line-w} and the main features of its equilibrium density, including the connections between the microscopic interactions and the corresponding formation of Pareto tails. In fact, unlike  the classical Bolzmann equation of rarefied gas dynamics, where the Maxwellian equilibrium density is easy to find, in the economic context, except in some simple cases \cite{BaTo,BaTo2,Sl04}, a precise analytic description of the emerging equilibria in a kinetic model of the Boltzmann equation is very difficult to obtain.

Going back to the possible weakness of the Maxwellian assumption, a careful analysis of the  economic transaction \fer{lin} allows us to conclude that the choice of a constant collision kernel leads to consider as possible also interactions which human agents would exclude \emph{a priori}. This is evident for interactions in which the wealth of the agent trading with the market is equal to zero (or extremely small). In this case, the outcome of the trade results in a net loss of money for the market, and it seems difficult to justify that an agent of the market would accept to trade. Likewise, this is true if the agent that trades with a certain amount of wealth, does not receive (excluding the risk) some wealth back from the market. In other words, trades in which $w$ or $v$ are equal to zero or extremely small should be excluded by the trading agents. On the contrary, the possibility to receive a consistent amount of wealth from the market, or for a market agent the possibility to trade with agents that have a consistent wealth needs to be considered more probable.

Hence, in the economic setting, it seems natural to consider collision kernels that select this behavior. A simple but consistent assumption is to define 
 \be\label{Ker}
 K(v,w) = \kappa\cdot \left( vw\right)^\delta,
  \ee
for some constants $0 < \delta \le 1$ and $\kappa >0$.  This kernel, which is clearly different from the collision kernel of elastic  particles, excludes the economic transactions in which one of the agents has no wealth to put on the game, and enhances transactions in which the amount of money of both agents is conspicuous. 

By taking into account this new assumption, we consider  in  the following that the wealth density satisfies the linear kinetic model 
 \begin{equation}
  \label{line}
 \frac{d}{dt}\int_{\R_+}f(t,w)\varphi(w)\,dw  = 
 \kappa\, \Big \langle \int_{\R_+\times \R_+} \left( vw\right)^\delta \left( \varphi(w^*)-\varphi(w) \right) f(t,w)\mathcal E(v)
\,dv\,dw\Big \rangle.
 \end{equation} 
The model includes the standard Maxwellian linear kinetic model, which is obtained for $\delta = 0$.
While the presence of the collision kernel  is more realistic from a modeling point of view, it introduces  additional difficulties, not present in the original Maxwellian assumption. This is evident for example when computing the evolution of moments, which, as it happens for the classical Boltzmann equation, obey to equations which are not in close form. 

Before studying on a general level the consequences of the introduction of the collision kernel  \fer{Ker}, we illustrate through the simple example of the \emph{pure gambling} model \cite{BaTo},  how  the equilibrium density is modified with respect to the one of the original  kinetic equation with a constant kernel.  To achieve this result, we will take essential advantage from the fact that the bilinear Boltzmann equation for the pure gambling model considered in Ref. \cite{BaTo} allows for an explicit derivation of  the steady state density in various situations.

Subsequently, we shall investigate the consequences of the presence of the interaction kernel at the Fokker--Planck level.  In fact, the Fokker--Planck description provides both further insights into the large-time behavior of the solution to the kinetic equation \fer{line}, and a more accessible description of the possible stationary states. 

%%%%%%%%%%%%%%%%%%%%%%%%%%%%%%%%%%%%%%%%%%%%%%%%%%%%

\section{The pure gambling with non Maxwellian collision kernel}\label{gambling}

The nonlinear kinetic equation of Boltzmann type which describes the evolution of wealth in a pure
gambling process was studied in Ref. \cite{BaTo}. In a pure gambling process \cite{DY00}, the entire sum of wealths of two agents is up for gambling, and randomly shared between the agents. In its original version, the randomness is introduced into the model through a parameter $\omega$ which is a random number drawn
from a probability distribution in $(0, 1)$. In general it is assumed that $\var$ is independent of a pair of agents, so that a pair of agents do not share the same fraction of  wealth  when they interact repeatedly.  Moreover, it is usually assumed that the game is fair, which can be obtained by taking the random variable $\omega$  symmetric with respect to the value $1/2$ (which implies that $\omega$ and $1-\omega$ are identically distributed).  Let $(v,w)$ denote the amounts of wealth played by the agents, and $(v^*,w^*)$ the post-trade wealths. Then the pure gambling is described by the interaction
 \be\label{gamb}
 v^* = \omega(v+w), \quad w^* = (1-\omega)(v+w).
 \ee
Interaction \fer{gamb} is pointwise conservative, namely
 \[
 v^* + w^* = v+w.
 \]
An interesting variant of the pure gambling model was introduced in Ref. \cite{BaTo}. To take into account the role of an external entity (like a bank, or, more in general, the market opportunities), both trading agents have the opportunity to gain (or to loose). This result is achieved by considering a pair of independent and identically distributed random variables, say $\omega_1$ and $\omega_2$, symmetric with respect to the value $1/2$, and to use them to describe the interaction. In this second case the result of the interaction is
 \be\label{gamb2}
 v^* = \omega_1(v+w), \quad w^* = \omega_2(v+w).
 \ee
Note that, unlike the interaction \fer{gamb}, \fer{gamb2} is conservative only in the mean. This means that, while in general 
 \[
 v^* + w^* \not= v+w,
 \]
equality holds in mean sense
 \[
 \langle v^* + w^*\rangle = v+ w.
 \]
For trading rules as in \fer{gamb}, and in presence of the interaction kernel \fer{Ker} the wealth density $f(t,w)$ of the agents system, as shown in Ref. \cite{BaTo},  satisfies a bilinear Boltzmann-like equation that in weak form reads
 \begin{equation}
  \label{line-g}
 \frac{d}{dt}\int_{\R_+}f(t,w)\varphi(w)\,dw  = 
 \kappa\, \Big \langle \int_{\R_+\times \R_+} \left( vw\right)^\delta \left( \varphi(w^*)-\varphi(w) \right) f(t,w)f(t,v)
\,dv\,dw\Big \rangle.
 \end{equation}  
Note that the equation considered in Ref. \cite{BaTo} was of the Maxwellian type, thus corresponding to the choice $\delta =0$. It is a simple exercise to show that the solution to equation \fer{line-g}, for any value of the constant 
$\delta$, satisfies the conservation of mass and momentum.  For the equation with $\delta =0$ the
analytical form of the steady states $f_\infty$ is found for various realizations of the random fraction of the sum which is shared to the agents.  Among
others, Gibbs distribution appears as a steady state in case of a uniformly distributed random fraction, while Gamma distribution appears for a
random fraction which is Beta distributed. It is immediate to verify that, setting
  \[
  g_\infty(w) = w^\delta f_\infty(w)
  \]
 $g_\infty$ solves
  \be\label{equ}
  \Big \langle \int_{\R_+\times \R_+}  \left( \varphi(w^*)-\varphi(w) \right) g_\infty(w)g_\infty(v)
\,dv\,dw\Big \rangle =0.
  \ee 
  Therefore, for any given random variable $\omega$, $g_\infty$ coincides with the stationary solution to the pure gambling Maxwellian model studied in Ref. \cite{BaTo}, corresponding to the same choice of $\omega$. Assume that the steady state is a probability density of unit mean. Hence, if $\omega$ is uniformly distributed in $(0,1)$ \cite{BaTo},  $g_\infty$ is shown to be exponentially (Gibbs) distributed with 
  \[
  g_\infty(w) = e^{-w}.
  \]
 Consequently, the steady state solution of the non Maxwellian model \fer{line-g} of unit mean, for any given $\delta <1$, is given by the Gamma density 
  \be\label{g1}
  f_\infty(w) = \frac{(1-\delta)^{1-\delta}}{\Gamma(1-\delta)} w^{-\delta} \exp\left\{ -(1-\delta)w \right\}.
  \ee
As discussed in Ref. \cite{BaTo},  some insight on the consequences of the choice $\delta >0$ can be gained by the simple computation of the variance of $f_\infty$. It holds
\[
Var(f_\infty)=\int_{\R^+}v^2f_\infty(v) dv -1= \frac{1}{1-\delta}.
\]
Hence the variance (the spreading) increases as $\delta$ increases. In particular, measures of the inequality of the wealth distribution, such as the Gini coefficient, increase for increasing $\delta$ , and tend to blow up as $\delta \to 1$.
 
This behavior is not unexpected. Indeed, in the original pure gambling model, corresponding to $\delta =0$, also agents with no wealth can play, and gain wealth from the gambling, thus moving away from their unlucky condition. Clearly, this is  a purely abstract model. If we want to adapt the model to the human behavior, we are forced to eliminate this non realistic gambling. It is clear that agents with positive wealth  would never accept to trade with agents with no wealth (or an extremely small wealth), knowing in advance that the only possibility for them is to be looser after gambling. This clearly implies that the percentage of agents with wealth close to zero can only increase with $\delta$ increasing, since they are automatically excluded from the game by the presence of the collision kernel  $K$.  The same consequences are shown to hold for any other choice of the random variable $\omega$ considered in Ref. \cite{BaTo} and leading to an explicit equilibrium density. 

The situation described by an interaction of type \fer{gamb2} is different. As discussed in Ref. \cite{BaTo},  the case in which the gambling game is only \emph{conservative-in-the-mean}  was shown to lead to an explicit heavy tailed inverse Gamma distribution. Following Ref. \cite{BaTo}, let us consider the pair of random variables $\omega_i$, $i =1,2$ given by
 \[
 \omega_i = \frac 1{4\vartheta_i} , \qquad i =1,2,
 \]
where, for a given $a > 1$ the random variable $\vartheta_i$, $i =1,2$  is a $Beta(a+1/2,a- 1/2)$ random variable.
Then, the solution to equation \fer{equ} of unit mean is an  inverse-Gamma distribution  of shape parameter $a$ and scale parameter $a - 1$, that is
 \be\label{gam}
g_\infty(w) = \frac{(a-1)^a}{\Gamma(a)}\frac{ \exp\left\{ -\frac{ a-1}w\right\}}{w^{1+a}},
\ee
which is peaked around the mean value 1 and has heavy tails, in that it decays at infinity like $w^{a+1}$. Consequently, in presence of the collision kernel $K$ the steady state changes into 
 \be\label{pp}
 f_\infty(w) = \frac{(a+\delta-1)^{a+\delta}}{\Gamma(a+\delta)}\frac{ \exp\left\{ -\frac{ a+\delta-1}w\right\}}{w^{1+a+\delta}},
  \ee
 It follows that $f_\infty$ has unit mean and, provided $a+\delta >2$ its variance has the value
  \[
  Var(f_\infty)= \frac{1}{a+\delta -2}.
  \]
  Hence, in contrast with the result of the first gambling model, here the variance of the steady state density is decreasing as $\delta$ increases, thus leading to a fairer society in presence of a larger value of the parameter $\delta$.   
 To summarize, the addiction of a variable collision kernel of type \fer{Ker} in the pure gambling model, for which exact solutions are available, enlightens a marked difference between the two interactions \fer{gamb}  and \fer{gamb2}.  In the first case, the presence of the kernel leads to a variance that increases with $\delta$. In this case, the presence of the kernel enhances its effect by inducing a strong variation of the population with wealth close to zero. In the second case, the presence of the kernel enhances its effects also on agents trading big amounts of wealth, by inducing a reduction of the number of very rich agents, which leads to a higher value of the Pareto index.

%%%%%%%%%%%%%%%%%%%%%%%%%%%%%%%%%%%%%%%%%%%%%%%%%%%%%%%%%
\section{Fokker-Planck description of the non-Maxwellian model}\label{quasi}

In this Section, we illustrate the main steps leading from equation \fer{line} to its Fokker--Planck limit. 
The relationship between the kinetic equation \fer{line} and its Fokker--Planck counterpart is obtained by resorting to the well-known \emph{grazing} asymptotic.
As exhaustively explained in Ref. \cite{FPTT17}, this asymptotic procedure is a well-consolidated technique which has been first developed for the classical Boltzmann equation \cite{Vi,vil2,Vil02}, where it is known under the name of \emph{grazing collision limit}. In the one-dimensional setting, this asymptotic procedure has been fruitfully applied to  the dissipative versions of Kac caricature of a Maxwell gas \cite{Furioli2012}, introduced in Ref.~ \cite{PTo}. 

Since this procedure is fully described in details in Ref. \cite{FPTT17}, we give below details only when there are marked differences with respect to the derivation for Maxwellian models. First of all, to avoid inessential difficulties, we will assume that the market density $\mathcal E$ has a certain number of moments bounded, more precisely
 \be\label{mo3}
 M_\alpha = \int_{\R_+} v^\alpha \mathcal E(v)\, dv < +\infty, \quad 0 \le \alpha \le 4.
 \ee
Among observable quantities,  by letting $\varphi=1$ in \fer{line} one shows that the mass is  conserved. Therefore, if the initial density is of unit mass, $f(t,w)$ remains a probability density at each subsequent time. Besides the mass,  the first representative moments to be studied are the mean value of the density $f(t,w)$, as well as its variance. In what follows, let
 \be\label{momi}
 m_\alpha(t) = \int_{\R_+}w^\alpha \,f(t,w)\,dw, \quad \alpha >0.
 \ee
By choosing $\varphi(w) = w$ in \fer{line} and remarking that \fer{lin} implies
 \[
 \langle w^* - w \rangle = \lambda(v-w) ,
 \]
we obtain 
 \begin{equation}
  \label{m-1}
 \frac{d}{dt}m_1(t)  = 
 \kappa \int_{\R_+ \times \R_+}(vw)^\delta\lambda(v-w)\,  f(t,w)\mathcal E(v)
\,dv\,dw.
 \end{equation}
 While the evolution of the mean value is not in closed form, it can be easily proven that the mean value remains bounded in time if it is so initially. Indeed,  the integral on the right-hand side is uniformly bounded in time from above. This follows from the inequality
 \be\label{222}
(vw)^\delta(v-w) \le v^{1+2\delta}.
 \ee
Inequality \fer{222} clearly holds if $v=0$ or $w= 0$. If $v>0$ and $w >0$,  \fer{222} is equivalent to the obvious inequality
 \[
 \left(\frac vw\right)^{1+\delta} +1 \ge  \frac vw. 
 \]
Using \fer{222} we obtain the upper bound
 \[
 \int_{\R_+ \times \R_+}(vw)^\delta(v-w)\,  f(t,w)\mathcal E(v)\,dv\,dw \le \int_{\R_+ \times \R_+}v^{1+2\delta}\,  f(t,w)\mathcal E(v)\,dv\,dw = M_{1+2\delta},
 \]
which implies, for any $t >0$
 \[
 m_1(t) \le m_1(0) + \kappa\lambda M_{1+2\delta}\, t.
 \]
A better estimate can be obtained by resorting to
 Jensen's inequality. Since
  \[
  \int_{\R_+ \times \R_+}(vw)^\delta(v-w)\,  f(t,w)\mathcal E(v)\,dv\,dw= - \int_{\R_+}\left( w^{1+\delta}M_\delta - w^\delta M_{1+\delta}\right)\,  f(t,w)\,dw,
  \]
 by using
  \[
 \left(\int_{\R_+}w \,f(t,w)\,dw\right)^{1+\delta} \le \int_{\R_+}w^{1+\delta} \,f(t,w)\,dw,
 \]
 and
\[
 \left(\int_{\R_+}w^\delta \,f(t,w)\,dw\right)^{1/\delta} \le \int_{\R_+}w \,f(t,w)\,dw,
 \]
 we obtain
  \begin{equations}
  \label{m33}
 \frac{d}{dt}m_1(t) &\le \kappa\,\lambda\left(m_1(t)^\delta M_{1+\delta} - m_1(t)^{1+\delta} M_\delta \right)\\
& = \kappa\,\lambda\, m_1(t)^\delta \left(M_{1+\delta} - m_1(t) M_\delta \right).
 \end{equations}
 Clearly, \fer{m33} implies the bound
  \be\label{m-bound}
  m_1(t) \le \max\left\{ m_1(0); \frac{M_{1+\delta}}{M_\delta}\right\},
  \ee
 namely the uniform boundedness of the mean value. Likewise, since
 \[
 \langle {w^*}^2 - w^2 \rangle = (\sigma +\lambda^2-2\lambda) w^2 +2\lambda(1-\lambda)vw  +\lambda^2 v^2,
 \]
 we obtain
 \begin{equations}
  \label{m-v}
 &\frac{d}{dt}m_2(t) =\\
 & 
 \kappa \int_{\R_+ \times \R_+}(vw)^\delta\left[(\sigma +\lambda^2-2\lambda) w^2 +2\lambda(1-\lambda)vw  +\lambda^2 v^2\right] \,  f(t,w)\mathcal E(v)
\,dv\,dw .
 \end{equations} 
 Hence, if the constants $\sigma$ and $\lambda$ satisfy
  \be\label{m-2}
  \sigma +\lambda^2< 2\lambda, 
  \ee
so that the coefficient of $w^2$ into the integral in \fer{m-v} is negative, proceeding as before with Jensen's inequality we conclude that the second moment remains bounded if it is so initially, and its time derivative satisfies the inequality
 \begin{equations}
  \label{m34}
 &\frac{d}{dt}m_2(t) \le \\
 &\kappa\, m_2(t)^{\delta/2} \left(-(2\lambda-\sigma -\lambda^2) M_\delta  m_2(t) + M_{2+\delta} + 2\lambda(1-\lambda)M_{1+\delta}m_2(t)^{1/2} \right).
 \end{equations}
 Since the right-hand side of inequality \fer{m34} is a second order equation in the unknown $m_2(t)^{1/2}$, and the coefficient of the square is negative, it follows that positivity of the right-hand side holds if and only if $m_2(t)$ does not cross the bounded value
 \[
\bar m_2 = \left\{\frac{\lambda(1-\lambda)M_{1+\delta} + \sqrt{(\lambda(1-\lambda)M_{1+\delta})^2 + M_\delta M_{2+\delta}(2\lambda -\sigma-\lambda^2)}}{M_\delta(2\lambda -\sigma-\lambda^2)} \right\}^2.
 \]
 Finally, if condition \fer{m-2} holds, we get
 \be\label{m2-bound}
  m_2(t) \le \max\left\{ m_2(0); \bar m_2\right\}.
  \ee
 Note that the boundedness of the moments of the distribution $\mathcal E$ is enough to guarantee that the moments at the first two orders of the solution to equation \fer{line-w} remain bounded at any time $t >0$, provided that they are bounded initially.

The previous argument can be used to prove, at the price of an increasing number of computations,  that moments of order $n\ge2$ are bounded, provided that they are bounded initially, any time the coefficient of the higher order term in the wealth variable $w$ in the expression
$\langle (w^*)^n - w^n \rangle $ is negative. Since the coefficient is equal to
 \[
 \langle (1-\lambda+ \eta)^n \rangle -1, 
 \]
this condition establishes a relationship between $\lambda$ and the moments of the random variable $\eta$. In the rest of this Section, we choose $\lambda$ and $\eta$ in such a way that the coefficient is negative for $n =3$, which implies that the moments of the solution to the kinetic equation \fer{line} are uniformly bounded up to the order three.
 
Let us suppose now that the interaction \fer{lin} produces a very small mean change of the wealth. This can be easily achieved by introducing the scaling
 \be\label{scal}
\lambda \to \e \,\lambda,\quad \eta \to \sqrt\e \, \eta,
 \ee
where $\e$ is a small parameter, $\e \ll 1$. Note that the scaling has been chosen to maintain the relationship between $\lambda$ and $\sigma$ as given by inequality  \fer{m-2}. This scaling will produce a small variation of the mean value \fer{m-1}
\[
\frac{d}{dt}m_1(t)  = 
 \e \, \kappa \int_{\R_+ \times \R_+}(vw)^\delta\lambda(v-w)\,  f(t,w)\mathcal E(v)
\,dv\,dw.  
 \]
To observe an evolution of the mean value independent of $\e,$we  resort to a scaling of time. Letting  $t \to  \e t$,  the evolution of the average value satisfies
\[
 \frac{d}{dt}m_1(t)  = 
 \kappa \int_{\R_+ \times \R_+}(vw)^\delta\lambda(v-w)\,  f(t,w)\mathcal E(v)
\,dv\,dw
 \]
namely the same evolution law for the average value of $f$  given by \fer{m-1}. Indeed, if we assume that the interactions are scaled to produce a very small change of wealth, to observe an evolution of the mean value independent of the smallness, we need to wait enough time to restore the original evolution. 

With this scaling, condition \fer{m-2} becomes
  \be\label{scala}
 \e\, \sigma +\e^2\lambda^2< 2\e\,\lambda.
  \ee
 Hence, \fer{m-v} takes the form
 \begin{equations}
  \label{msca}
 &\frac{d}{dt}m_2(t) \\
 & 
= \e\, \kappa \int_{\R_+ \times \R_+}(vw)^\delta\left[(\sigma +\e\,\lambda^2-2\lambda) w^2 +2\lambda(1-\e\lambda)vw  +\e \,\lambda^2 v^2\right] \,  f(t,w)\mathcal E(v)
\,dv\,dw \\
&= \e\, \kappa \int_{\R_+ \times \R_+}(vw)^\delta\left[(\sigma -2\lambda) w^2 +2\lambda vw \right] \,  f(t,w)\mathcal E(v)
\,dv\,dw +  R_\e(t).
 \end{equations}
In \fer{msca} the remainder term is
 \be\label{rem}
 R_\e(t)= \e^2\, \kappa \int_{\R_+ \times \R_+}(vw)^\delta \lambda^2(v-w)^2 \,  f(t,w)\mathcal E(v)
\,dv\,dw.
 \ee
After the time scaling, the second order moments satisfies the equation
 \[
\frac{d}{dt}m_2(t)
=  \kappa \int_{\R_+ \times \R_+}(vw)^\delta\left[(\sigma -2\lambda) w^2 +2\lambda vw \right] \,  f(t,w)\mathcal E(v)
\,dv\,dw +  \frac 1\e R_\e(t).
 \]
Since the remainder vanishes at the order $\e^2$ as $\e \to 0$, one obtains in the limit a closed form for the evolution of the second moment. 

The final step is to consider, under the simultaneous scaling of wealth (given by \fer{scal}) and time, the evolution of a general observable.
Given a smooth function $\varphi(w)$, let us expand in Taylor series $\varphi(w^*)$ around $\varphi(w)$. It holds
 \[
\langle w^* -w \rangle = \e \,\lambda(v-w); \quad  \langle (w^* -w)^2\rangle =  \e^2 \,\lambda^2(v-w)^2 + \e \sigma w^2.
 \]
Therefore, in terms of powers of $\e$,  we easily obtain the expression
 \[
\langle \varphi(w^*) -\varphi(w) \rangle =  \e \left( \varphi'(w)\lambda(v-w) + \frac 12 \, \varphi''(w)\, \sigma w^2 \right) + R_\e (v,w),
 \]
where the remainder term $R_\e$  vanishes at the order $\e^{3/2}$ as $\e \to 0$ \cite{FPTT17} 

Substitution into equation \fer{line} and scaling in time  give
\begin{equations}\label{fp}
&\frac{d}{dt}\int_{\R_+}f(t,w)\varphi(w)\,dw \\
& = 
 \kappa\,  \int_{\R_+\times \R_+} \left( vw\right)^\delta \left( \varphi'(w)\lambda(v-w) + \frac 12 \, \varphi''(w)\, \sigma w^2  \right) f(t,w)\mathcal E(v)
\,dv\,dw \\
&+\frac  \kappa \e\,\int_{\R_+\times \R_+} \left( vw\right)^\delta R_\e (v,w)f(t,w)\mathcal E(v)
\,dv\,dw.
 \end{equations}
Letting $\e \to 0$, and evaluating the integrals with respect to the market density $\mathcal E(v)$, shows that in consequence of the scaling \fer{scal} the weak form of the kinetic model \fer{line} is well approximated by the weak form of a linear Fokker--Planck equation (with variable coefficients)
\begin{equations}
  \label{m-13}
 & \frac{d}{dt}\int_{\R_+}\varphi(w) \,f(t,w)\,dw = \\
  & \kappa \int_{\R_+} \left( \varphi'(w)\,\lambda\, w^\delta(M_{1+\delta} - M_\delta w) + \frac 12 \varphi''(w) \sigma\, M_\delta w^{2+\delta} \right) f(t,w)\, dw. 
 \end{equations}
By choosing $\varphi(w) = 1$ into \fer{m-13} we show that the mass density is preserved in time, so that, for any given time $t >0$
 \be\label{mass1}
 \int_{\R_+} \,f(t,w)\,dw =  \int_{\R_+} \,f(t,w=0)\,dw .
 \ee
Therefore, if the initial value is given by a probability density function, the (possible) solutions of the Fokker--Planck equation \fer{m-13} remain  probability densities for all subsequent times.

If boundary conditions on $w=0$ and $w=+\infty$ are added, such that the boundary terms produced by the integration by parts vanish, equation \fer{m-13} coincides with the weak form of the Fokker--Planck equation
 \begin{equation}\label{FP22}
 \frac{\partial f}{\partial t} = \kappa\, M_\delta\left[ \frac\sigma 2 \frac{\partial^2 }{\partial w^2}
 \left(w^{2+\delta}f \right )+ \lambda\,
 \frac{\partial}{\partial w}\left( w^\delta(w - M_{1+\delta}/M_\delta ) f\right)\right].
 \end{equation}
Without loss of generality, we will simplify equation \fer{FP22} by assuming
 \be\label{simp}
 \frac{\kappa\, \sigma \, M_\delta}2 = 1, \quad \frac{M_{1+\delta}}{M_\delta} = m, \quad \frac{2\lambda}\sigma = \mu.
 \ee
Thus, the resulting Fokker--Planck equation takes the form
\begin{equation}\label{FP2}
 \frac{\partial f}{\partial t} =  \frac{\partial^2 }{\partial w^2}
 \left(w^{2+\delta}f\right )+ 
\mu\, \frac{\partial}{\partial w}\left( w^\delta(w - m ) f\right).
 \end{equation}
As exhaustively discussed in Ref. \cite{FPTT17} (cf. also the analysis of Section \ref{large-time}),  the right boundary conditions that guarantee mass conservation are the so-called \emph{no--flux} boundary conditions, given by 
\be\label{bc1}
\left. \frac{\partial }{\partial w}\left( w^{2+\delta} f(t,w)\right) + \mu \, w^\delta (w-m) f(t,w) \right|_{w=0,+\infty} = 0, \quad t>0.
\ee
\begin{rem} {\rm It is interesting to remark that the presence of the collision kernel in the Boltzmann equation \fer{line} results in a modification of both the diffusion and the drift terms in the Fokker--Planck equation. These modifications cancel by choosing $\delta = 0$, that corresponds to the Maxwellian case studied in Ref. \cite{CoPaTo05}}.
\end{rem}

 With respect to the Maxwellian case, the presence of $\delta >0$ in \fer{FP2} does not modify the shape of the equilibrium density, that can be easily recovered by solving the first-order differential equation 
 \be\label{ste}
 \frac{\partial }{\partial w}
 \left(w^{2+\delta}f\right ) = 
-\mu\,\left( w^\delta(w - m ) f\right).
 \ee
Using  $g(w) = w^{2+\delta}f(w)$ in \fer{ste} as unknown function, shows that the unique equilibrium density of unit mass is the inverse Gamma function
 \be\label{new-eq}
 f_\infty^\delta (w) = \frac{(\mu\, m)^{1+\delta+ \mu}}{\Gamma(1+\delta + \mu)}\frac{\exp\left(-\frac{\mu\, m}{w}\right)}{w^{2+\delta +\mu}}.
 \ee
 Hence, the presence of the collision kernel is such that the Pareto index in the equilibrium density of the target Fokker--Planck equation is an inverse Gamma density with the tail exponent increased by the amount $\delta$. This difference is in agreement with the result found in Section \ref{gambling} for the pure gambling model with interactions of type \fer{gamb2} and a non-Maxwellian collision kernel.  By discarding economic interactions involving agents with very small wealth, and enhancing interactions between very rich agents we surprisingly generate a less unequal distribution of wealth in the society. Indeed, provided $\delta + \mu >1$, the variance of the inverse Gamma density \fer{new-eq} is equal to
 \[ 
 Var\left( f_\infty^\delta \right) = \frac{(\mu m)^2}{(\delta+\mu)^2(\delta + \mu -1)},
 \]
so that the variance decreases as $\delta$ increases. 

\begin{rem}\label{mu-b}
{\rm  In the rest of the paper, we will always suppose that the value of the constant $\mu$ appearing in front of the drift term in the Fokker--Planck equation \fer{FP2} is bigger than one. In this case, for all values of $\delta >0$, the variance of the stationary solution \fer{new-eq} is a bounded quantity. In particular, since $\delta \le 1$
 \[
 \lim_{w \to +\infty} w^{2+\delta} f_\infty^\delta(w) = 0.
 \]
}
\end{rem}

The existence of the unique equilibrium density of unit mass, given by  \fer{new-eq}, allows us to write the Fokker--Planck equation \fer{FP2} in other equivalent formulations. 
%To simplify notations, for any value of the constant $\delta$,  let us denote the equilibrium density \fer{new-eq} simply by $f_\infty$.
Then, since $f_\infty^\delta$ solves \fer{ste},
 we can write
  \[
\frac{\partial }{\partial w}\left(w^{2+\delta}
 f \right) + \mu w^\delta(w -m)\,f =   w^{2+\delta}f \left( \frac{\partial }{\partial w} \log\left( w^{2+\delta}
 f \right)  + \mu\,\frac{w-m}{w^2}\right)=
 \]
 \[
 w^{2+\delta} f \left( \frac{\partial }{\partial w} \log\left( w^{2+\delta}
 f \right)- \frac{\partial }{\partial w} \log\left( w^{2+\delta}
 f_\infty^\delta \right)\right) =  
  \]
  \[
 w^{2+\delta} f  \frac{\partial }{\partial w} \log\frac f{f_\infty^\delta}=  w^{2+\delta}f_\infty^\delta \frac{\partial }{\partial w}\frac f{f_\infty^\delta}.
 \]
Hence, we can write the Fokker--Planck equation \fer{FP2} in the equivalent form
 \be\label{FPalt}
  \frac{\partial f}{\partial t} = \frac{\partial }{\partial w}\left[  w^{2+\delta}f \frac{\partial }{\partial w} \log\frac f{f_\infty^\delta}\right],
 \ee
which enlightens the role of the logarithm of the quotient $f/f_\infty^\delta$, or in the form
 \be\label{FPal2}
  \frac{\partial f}{\partial t} = \frac{\partial }{\partial w}\left[ w^{2+\delta} f_\infty^\delta \frac{\partial }{\partial w} \frac f{f_\infty^\delta}\right].
 \ee 
In particular,   equation \fer{FPal2}  allows us to obtain the evolution equation for the quotient $F= f/f_\infty^\delta$. Indeed
 \[
  \frac{\partial f}{\partial t} = f_\infty^\delta  \frac{\partial F}{\partial t} = w^{2+\delta} f_\infty^\delta \frac{\partial^2 }{\partial w^2} \frac f{f_\infty^\delta} +\frac{\partial }{\partial w} ( w^{2+\delta}f_\infty^\delta) \frac{\partial }{\partial w} \frac f{f_\infty^\delta}= 
 \]
 \[
 w^{2+\delta} f_\infty^\delta \frac{\partial^2 F }{\partial w^2}  -\mu w^\delta(w-m) f_\infty^\delta \frac{\partial F}{\partial w},  
 \]
which shows that $F= F(t,w)$ satisfies the equation
 \be\label{quo}
\frac{\partial F}{\partial t} =   w^{2+\delta} \frac{\partial^2 F }{\partial w^2}  -\mu w^\delta(w-m) \frac{\partial F}{\partial w}.
 \ee
 Equation \fer{quo} is usually known as the \emph{adjoint} of \fer{FP2}.
 
 The boundary conditions of  the adjoint form of the Fokker-Planck equation then follow from \fer{bc1}  and read
\be\label{bc2} 
\left.  w^{2+\delta} f_\infty^\delta(w) \frac{\partial }{\partial w} \frac {f(t,w)}{f_\infty^\delta(w)} \right |_{w=0, +\infty} = \left.  w^{2+\delta} f_\infty^\delta(w) \frac{\partial F(t,w) }{\partial w} \right |_{w=0,+\infty} =0.
\ee 
In consequence of the assumption made on $\mu$ (cf. Remark \ref{mu-b}) the boundary conditions \fer{bc2} are satisfied provided the derivative  $\frac{\partial F(t,w) }{\partial w}$ is bounded.
 %%%%%%%%%%%%%%%%%%%%%%%%%%%%%%%%%%%%%%%%%%

\section{Large-time behavior of the Fokker--Planck equation} \label{large-time}

In Section \ref{quasi}, we introduced the kinetic model \fer{line}  of Boltzmann type. While this model was described at the microscopic level by the binary trading interactions among agents proposed in Ref. \cite{CoPaTo05}, it was additionally coupled with a non Maxwellian kernel. In this way,  the model has the further property to select or discard interactions in accord with elementary economic principles. Consequently, it is reasonable to assume that it provides a better approximation to the relaxation process of the wealth distribution of a multi-agent society towards equilibrium. 
Similarly to the situation studied in Ref. \cite{CoPaTo05}, the non-Maxwellian  kernel of the Boltzmann type kinetic model \fer{line}  leads, in the grazing collision limit, to a linear Fokker--Planck equation \fer{FP2}. While maintaining the shape of the equilibrium density of the Maxwellian case, this equation is characterized by the presence of a further power $w^\delta$, with $\delta >0$,  in both the diffusion and drift coefficients. 

The main advantage of the Fokker--Planck description is related to the possibility to express analytically the steady state, and to resort to various mathematical methods to analyze the rate of convergence to equilibrium of its solution.
In  a related paper \cite{FPTT17}, we enlightened the main difficulties encountered in order to study rates of convergence towards equilibrium of  the solution to Fokker--Planck type equations with variable coefficients of diffusion and linear drift by using entropy methods. On the other hand, these methods appear very natural to apply, since they allow for precise results if the classical Fokker--Planck equation with constant coefficient of diffusion and linear drift is dealt with. 

Among the various models considered in Ref. \cite{FPTT17}, equation \fer{FP2} with $\delta =0$ was included as a leading example. The presence of a  diffusion term with variable diffusion coefficient led to the conclusion that in general the entropy methods fail to give exhaustive results. Related findings in this direction, directly connected to the differential inequalities which are classically used to control convergence towards equilibrium, were obtained before in Ref. \cite{MJT}. 

In this section, we will show that the modification produced in the Fokker--Planck equation by the choice  of a non-Maxwellian kernel,  in agreement with the economic behavior of agents,  allows for a fundamental improvement in the large-time behavior of the solution, that can be shown to converge exponentially fast towards equilibrium in relative entropy at explicit rate. 

%%%%%%%%%%%%%%%%%%%%%%%%%%%%%%%%%%%%%%%%%%%%%%%%%%%%

\subsection{Existence results}

 Fokker-Planck equation \fer{FP2} is included in the class  of one-dimensional Fokker--Planck equations that can be fruitfully written in divergence form as
 \be\label{FFPP}
  \frac{\partial f(t,x)}{\partial t} = \frac{\partial }{\partial x}\left[ \frac{\partial }{\partial x}\left(a(x)
 f(t,x)\right) + b(x) \,f(t,x)\right],
 \ee
where $x\geq 0$ and the diffusion coefficient $a(x)$ is a nonnegative function, strictly positive on the interior of the domain.

The initial-boundary value problem for equation \fer{FFPP} in  $\R_+$ when $a(x) = \alpha x$ and $b(x) = -(\beta x +\gamma)$, with $\alpha, \beta$ and $\gamma$ constants, has been first studied by Feller in Ref. \cite{Fel1} at the beginning of the fifties of last century, resorting to the powerful tool of semigroup theory. In a second seminal paper \cite{Fel2}, written one year later, Feller extended his results to a larger class of diffusion  and drift  coefficients. Equation \fer{FFPP} was complemented in Ref. \cite{Fel2} with different types of boundary conditions, leading to different results of existence and uniqueness. The analysis of Feller took into account both equation \fer{FFPP} and its adjoint equation
\be\label{FPad}
  \frac{\partial F(t,x)}{\partial t} = a(x) \frac{\partial^2 F(t,x)}{\partial x^2}
 - b(x) \, \frac{\partial F(t,x) }{\partial x}.
 \ee
In Ref. \cite{Fel2},  existence and uniqueness of the solution to equations \fer{FFPP}, \fer{FPad}, were studied by assuming that $a(x), a'(x) $ and $b(x)$ were continuous, but not necessarily bounded,  in the interior of the domain, where $a(x) >0$. Further, the boundaries of the domain were classified by looking at the integrability properties of the function
 \be\label{psi}
 \Psi(x) = \exp \left\{  \int_{x_0}^x b(y) a^{-1}(y)\,dy \right\},
 \ee
 where $0 < x_0 <\infty$. 
Feller's analysis applies to the Fokker--Planck equation \fer{FP2}, where   $a(x)= x^{2+\delta}$, and $b(x) = \mu x^\delta(x-m)$.
Moreover, since the stationary solution \fer{new-eq} is directly related to \fer{psi}, and
 \[
 f_\infty^\delta(x) = C a^{-1}(x) \Psi^{-1}(x),
 \]
it is possible  to conclude that, in the language of Ref. \cite{Fel2},  $x=0$ is an entrance boundary for $0<\delta<1$ and a natural one for $\delta =1$, while  $x = +\infty$ is always an entrance boundary (cf. the discussion of Sec. 23 of Ref. \cite{Fel2}). 

In this situation, by applying the results of Ref. \cite{Fel2}, we obtain the following existence and uniqueness result. 

\begin{thm}[Feller  \cite{Fel2}] \label{Feller}
Let us consider the  initial value problem
\begin{equation}\label{bp-FP2}
\left\{
\begin{aligned}
& \frac{\partial f}{\partial t} =  \frac{\partial^2 }{\partial w^2}
 \left(w^{2+\delta}f\right )+  \mu\, \frac{\partial}{\partial w}\left( w^\delta(w - m ) f\right),\quad t>0,\ w\in (0,+\infty)\\
 & f(0,w)= f_0(w)
 \end{aligned}
 \right .
\end{equation}
where $\delta \in (0,1]$, $\mu$, $m$ are positive constants, and $f_0$ is a probability density in $\R_+$. Let moreover $f_0 \in \Sigma$, where
\[
\Sigma =\left\{
\begin{aligned}
& \phi \in C([0,+\infty))\cap C^2((0,+\infty)) : \\
&\lim_{w\to 0^+} \frac{\phi(w)}{f_\infty^\delta(w)} \in \R_+ \\
&\lim_{w\to +\infty} \frac{\phi(w)}{f_\infty^\delta(w)} \in \R_+ \\
& \lim_{w\to 0^+} \left( \frac{\partial}{\partial w}  \left(w^{2+\delta}\phi \right )+  \mu\,\left( w^\delta(w - m ) \phi \right ) \right )= 0\\
 & \lim_{w\to +\infty} \left( \frac{\partial}{\partial w}  \left(w^{2+\delta}\phi\right )+  \mu\,\left( w^\delta(w - m ) \phi \right )\right ) = 0
 \end{aligned}
 \right \}.
 \]
 Then there exists a positive solution $f(t,w)$ of  problem \eqref{bp-FP2}, which is unique  in 
 $C([0,+\infty), \Sigma)  \cap C^1([0,+\infty), L^1(\R_+))$.
 Moreover, $f(t,w)$ remains a probability density for all $t>0$.
 \end{thm}
 \begin{rem}
Note that the no--flux boundary conditions are contained in the definition of the space $\Sigma$ itself.
 \end{rem}
 \begin{proof}
 We make use of the relationship between a solution $f(t)$ of the Fokker--Planck equation \eqref{FP2} and a solution $F(t)$ of the adjoint equation \eqref{quo}.
Let $F_0= f_0/f_\infty^\delta$ the initial data for the initial value problem for equation \eqref{quo}. 
Then, the assumptions on $f_0$ translate into assumptions on $F_0$. 

Hence, $F_0 \in \widetilde \Sigma$ where
 \[
 \widetilde \Sigma =\left\{ \psi \in C([0,+\infty))\cap C^2((0,+\infty)) : 
 \begin{aligned}
% &\lim_{w\to 0^+} \psi(w) \in \R \\
&\lim_{w\to +\infty}  \psi(w) \in \R_+ \\
& \lim_{w\to 0^+}  \left(w^{2+\delta}f_\infty^\delta(w)\frac {\partial \psi(w)}{\partial w} \right )= 0\\
 & \lim_{w\to +\infty} \left( w^{2+\delta}f_\infty^\delta(w) \frac {\partial \psi(w)}{\partial w} \right ) = 0
 \end{aligned}
  \right \}.
 \]
 Feller's analysis shows that under these assumptions $F_0$ belongs to the domain of the operator $\Omega = w^{2+\delta} \frac{d^2}{d w^2}
 -  \mu\, w^\delta(w - m ) \frac{d}{d w}$ and Hille--Yosida theorem (cf. for example Ref. \cite{Paz}) applies. The Cauchy problem for \eqref{quo} with $F_0$ as initial data possesses therefore  a positive solution,
 unique  in $C([0,+\infty), \widetilde\Sigma ) \cap C^1([0,+\infty), C_b([0,+\infty)))$.
Consequently, $f(t,w)= F(t,w) f_\infty^\delta(w)$ is a positive solution of the  problem \eqref{bp-FP2}. Since $f_0$ belongs to the domain of the operator $\Omega^\star =  \frac{d}{d w} \left( \frac{d}{d w} \left(w^{2+\delta} \cdot\right )
 +  \mu\, w^\delta(w - m )\cdot \right )$, the solution $f(t)$ is unique in $C([0,+\infty), \Sigma ) \cap C^1([0,+\infty), L^1(\R_+))$. Feller further proved  that the $L^1$ norm is preserved, but this can be easily shown directly from the equation and from the boundary conditions contained in the definition of $\widetilde \Sigma$.
 \end{proof}
 \begin{rem}\label{ac}
 The solution $f(t)$ obtained by Feller's analysis is absolutely continuous with respect to the steady state $f_\infty^\delta$ for all $t\ge 0$.
 \end{rem}

\begin{rem}\label{desire}
The class of initial data in Theorem \ref{Feller} is quite restricted with respect to the natural one, which would contain all probability densities in $\R_+$. On the other hand, the properties of the solution to \fer{bp-FP2} guaranteed by Theorem \ref{Feller} allow us to investigate rigorously the large-time behavior of the solution, and to obtain exponential convergence to equilibrium in relative entropy at explicit rate. Consequently, the forthcoming analysis of the large-time behavior of the solution will be restricted to initial data as in Theorem \ref{Feller}.
 
We remark that in Ref. \cite{Fel2} Feller proved that for all $f_0 \in L^1(\R_+)$,  the Cauchy problem \eqref{bp-FP2}, posed in  $C([0,+\infty), L^1(\R_+))  \cap C^1([0,+\infty), L^1(\R_+))$ still has a  solution defined through the semigroup generated by the operator $\Omega^\star$. However, unlike the case of initial values in $\Sigma$, it is not proven that this solution still satisfies the boundary conditions for positive times.
While leaving to further research the possibility to extend the result of Theorem \ref{Feller} to general initial data in $L^1(\R_+)$,  we will show in the last Section of the paper that convergence to equilibrium at exponential rate follows also for a larger class of initial data by resorting to the result for initial densities in $\Sigma$. 
\end{rem}

%%%%%%%%%%%%%%%%%%%%%%%%%%%%%%%%%%%%%%%%%%%%%%%%%%%%

\subsection{An equivalent Fokker--Planck equation}\label{equiva}

A further interesting remark made by Feller in Ref. \cite{Fel2} was concerned with the possibility to introduce a transformation of variables to reduce the coefficient $a(x)$ in \fer{FFPP} and \fer{FPad} to a constant value. If the Fokker--Planck equation  \fer{quo} is considered, the transformation considered in Ref. \cite{Fel2} can be expressed by
\begin{equation}\label{G}
F(t,x)=G(t,y), \quad y = y(x)= \frac 2\delta \frac 1{x^{\delta/ 2}}.
\end{equation}
Note that \fer{G} implies
\be\label{dy}
\frac{dy}{dx}= - \frac 1 {x^{1+\delta/2}},
\ee
so that
\begin{equation}\label{newvariable}
\left(\frac{dy}{dx}\right )^2 =  \frac 1{x^{2+\delta}}.
\end{equation}
The change of variable in \fer{G} maps $\R_+$ into $\R_+$, and it is well-defined in the interior. Since
\[
\begin{aligned}
&\frac{\partial F(t,x)}{\partial x}= \frac{\partial G(t,y)}{\partial y} \frac{dy}{dx}\\
&\frac{\partial^2 F(t,x)}{\partial x^2}= \frac{\partial^2 G(t,y)}{\partial y^2} \left( \frac {dy}{dx}\right )^2 + \frac{\partial G(t,y)}{\partial y}  \frac {d^2y}{dx^2},
\end{aligned}
\]
equation \eqref{quo} transforms into
\begin{equation}\label{adj-G}
\frac{\partial G(t,y)}{\partial t}= \frac{\partial^2 G(t,y)}{\partial y^2} -
\frac{\partial G(t,y)}{\partial y} \left(\mu \, m \left(\frac \delta 2\right )^{ \frac 2 \delta -1} y^{\frac 2\delta -1} - \frac 2\delta \left(1+ \mu + \frac \delta 2\right )\frac 1 y\right ),
\end{equation}
which is such that the diffusion coefficient is equal to unity.

If we denote the drift term by
\be\label{dri}
W'(y)= \mu \, m \left(\frac \delta 2\right )^{ \frac 2 \delta -1} y^{\frac 2\delta -1} - \frac 2\delta \left(1+ \mu + \frac \delta 2\right )\frac 1 y,
\ee
equation \eqref{adj-G} takes the form
\be\label{adj-GW} 
\frac{\partial G(t,y)}{\partial t}= \frac{\partial^2 G(t,y)}{\partial y^2}   - W'(y)\frac{\partial G(t,y)}{\partial y}.
\ee
Equation \fer{adj-GW} is the adjoint of the  Fokker--Planck equation  
\be\label{FP-z}
\frac{\partial g(t,y)}{\partial t}= \frac{\partial^2 g(t,y)}{\partial y^2} +
\frac{\partial}{\partial y}\left(  W'(y)g(t,y)\right ),
\ee
still characterized by a  diffusion coefficient equal to one. Equation \fer{FP-z} has a steady state of the form
\be\label{SS-z}
g_\infty(y)= C e^{-W(y)}
\ee
where  $C>0$  and for $y_0>0$
\be\label{pot}
W(y)= \int_{y_0}^y W'(z) dz, \quad y>0.
\ee
 If we fix the mass of the steady state \fer{SS-z} equal to one, it is a simple exercise to reckon that \fer{SS-z} is a generalized Gamma density. We recall that the generalized Gamma is a probability density characterized in terms of a shape $\kappa>0$, a scale parameter $\theta >0$, and an exponent $\nu>0$, that reads \cite{Lie,Sta}
 \be\label{gg}
 h(y,\kappa,\nu, \theta) = \frac\nu{\theta^\kappa \Gamma \left( \kappa/\nu\right)}\, y^{\kappa -1} \exp\left\{ -\left( \frac y\theta\right)^\nu \right\}.
 \ee
For the steady state of the Fokker--Planck \fer{FP-z} the parameters  are given by  
 \be\label{para}
\nu = \frac 2\delta \ge 2, \quad \kappa =   \frac 2\delta(1+\delta +\mu) >2,  \quad  \theta = \frac 2{\delta \, (\mu m)^{\delta/2}}
 \ee
 with $\mu$ as in \eqref{simp}.
Equations \fer{FP-z} and \fer{adj-GW} are subject to boundary conditions derived from \fer{bc1} and \fer{bc2}, suitably modified according to the change of variables, which guarantee mass conservation. 

\begin{rem}\label{relaz}
We remark that the solutions $f(t,x)$ and $g(t,y)$ of equations \eqref{FP2} and \eqref{FP-z} are related by the change of variable \fer{G} so that
 \be\label{equiv-g}
 f(t,x) = g(t,y(x))\left| \frac{dy}{dx}\right|.
 \ee
 In particular, the same relation holds true between the generalized Gamma density defined in \fer{gg} and the inverse Gamma density \fer{new-eq}, so that
 \be\label{equiv}
 f_\infty^\delta(x) = g_\infty(y(x))\left| \frac{dy}{dx}\right|.
 \ee
 Moreover,  since $f(t)$ is absolutely continuous with respect to $f_\infty^\delta$ (see Remark \eqref{ac}), it follows that also $g(t)$ is absolutely continuous with respect to $g_\infty$.
\end{rem}

\begin{rem}\label{caso-no}
{\rm The case $\delta =0$ leads to a completely different behavior. In this case, the transformation \fer{G} becomes
\begin{equation}\label{G1}
F(t,x)=G(t,y), \quad y = y(x)= \log x,
\end{equation}
that implies
\be\label{dy1}
\frac{dy}{dx}=  \frac 1 {x}.
\ee
Therefore, the change of variable \fer{G1} maps $\R_+$ into $\R$, and the new Fokker-Planck equation with constant coefficient of diffusion is given by \fer{FP-z}, where now
 \be\label{St}
 W'(y) = 1+\mu-  \mu m e^{-y}, \quad y \in \R,
 \ee
so that
 \be\label{eq-no}
 g_\infty(y) = C \exp\left\{ -\left((1+\mu)y + \mu m e^{-y}\right)\right\}, \quad y \in \R.
 \ee
In the same way as  for the density \fer{gg}, the value of the constant $C$ that makes $g_\infty$  in \fer{eq-no} a probability density on $\R$ follows from the change of variable \fer{G1} applied to the equilibrium density \fer{new-eq}, with $\delta = 0$.  One obtains
 \[
 C = \frac{(\mu m)^{1+\mu}}{\Gamma(1+\mu)}.
 \]
Note that, in contrast to the case $\delta >0$,  $g_\infty$  is defined on the whole line, and it exhibits different decay rates at $y \to \pm\infty$. We will be back to further consequences of this behavior in the forthcoming Section.  
}
\end{rem}

%%%%%%%%%%%%%%%%%%%%%%%%%%%%%%%%%%%%%%%%%%%%%%%%%%%%
\subsection{Logarithmic Sobolev inequality and exponential decay of relative entropy}\label{deca}

The main result of  Section \ref{equiva} was to show that the Fokker--Planck type equation \fer{FP2} and its adjoint can be equivalently written as  Fokker--Planck equations with constant coefficient of diffusion and  nonlinear coefficient of  drift. This new Fokker--Planck equation is further characterized by a steady state in the form of a generalized Gamma density.  

Fokker--Planck equations  of type \fer{FP-z} have been introduced and studied in Ref. \cite{OV} in connection with the celebrated discovery by Bakry and \'Emery \cite{BE}, as a useful  mathematical tool to obtain logarithmic Sobolev inequalities for probability densities different from the standard Gaussian density. Given the equilibrium density $g_\infty$ of the Fokker--Planck equation \fer{FP-z}, let $H(g|g_\infty)$ denote the Shannon entropy functional of a probability density $g(y)$,  with $y \in \R_+$, relative to $g_\infty$, given by
\be\label{entr-class}
H(g|g_\infty) = \int_{\R_+} g(y) \log \frac{g(y)}{g_\infty(y)}\, dy.
\ee
Moreover, let $I(g|g_\infty)$ denote the Fisher information of the probability density $g$ relative to $g_\infty$,  defined as
\be\label{fisher-class}
I(g|g_\infty) = \int_{\R_+} \left(
\frac{d}{d y} \log  \frac {g(y)}{g_\infty(y)}\right )^2 g(y)\, dy.
\ee
Provided that the potential $W$ defined in \fer{pot} is uniformly convex, so that
 \be\label{log1}
 W''(y) \ge \rho >0,
 \ee
Bakry and \'Emery \cite{BE} proved that the probability measure $g_\infty$ satisfies a logarithmic Sobolev inequality with constant $\rho$. This corresponds to say that, for all probability measures $g$ absolutely continuous with respect to $g_\infty$, it holds
 \be\label{LS}
H(g|g_\infty) \le \frac 1{2\rho} I(g|g_\infty).  
 \ee
Inequality \fer{LS} was obtained in Ref. \cite{OV} by studying the evolution in time of the relative Shannon entropy $H(g(t)|g_\infty)$, where $g(t)$ is the solution to the Fokker--Planck equation \fer{FP-z}. 

In presence of boundary conditions that guarantee that the contribution on the boundary vanishes,  it is immediate to verify that, at any given time $t >0$, the relative Fisher information $I(g(t)|g_\infty)$ coincides with the entropy production at time $t$. In this case
 \be\label{EP}
 \frac d{dt}H(g(t)|g_\infty) = -I(g(t)|g_\infty).
 \ee
As a consequence, if inequality \fer{LS} holds  and since $g(t)$ is absolutely continuous with respect to $g_\infty$ (see Remark \eqref{relaz}), one easily concludes with the exponential convergence of the relative Shannon entropy to zero at the explicit rate $2\rho$.

\begin{rem} We note once again that, starting from initial data as in Theorem \ref{Feller}, the properties of the solution to the Fokker--Planck equation \fer{bp-FP2}, and in particular its behavior on the boundaries fully justify the computations on the solution to the new Fokker--Planck equation \fer{FP-z}, leading to \fer{EP}. 
\end{rem}
  
To verify the uniform convexity of $W(y)$, let us estimate $W''(y)$. The computations are immediate if $\delta =1$. Indeed, if we set $\delta =1$ in expression \fer{dri} we obtain
\[
W'(y)= \frac {m\mu }2 y - \left(\frac 32 +\mu \right ) \frac 2{y},
\]
that differentiating gives
\be\label{tau1}
W''(y)= \frac {m\mu}2  + \left(\frac 32 +\mu \right ) \frac 2{y^2}\geq \frac {m\mu}2 :=\rho(1).
\ee
When $0<\delta < 1$ simple but tedious computations give the bound
 \be\label{bbb}
 W''(y) \ge \frac \delta{2} (m\mu)^{\delta}\left(1 + \frac \delta{2} +\mu  \right)^{1-\delta} \frac{(2-\delta)^\delta}{(1-\delta)^{1-\delta}}\,\delta^{-2\delta}:=\rho(\delta).
 \ee
Note that, as $\delta \to 1$, one correctly obtains
 \[
 \frac \delta{2} (m\mu)^{\delta}\left(1 + \frac \delta{2} +\mu  \right)^{1-\delta} \frac{(2-\delta)^\delta}{(1-\delta)^{1-\delta}} \delta^{-2\delta} \to \frac {m\mu}2.
 \]
Finally, for any given value $0< \delta \le 1$, it follows that the generalized Gamma density \fer{gg} with parameters given by \fer{para}, satisfies the logarithmic Sobolev inequality \fer{LS} with
$ \rho = \rho(\delta)  >0$
 with $\rho(\delta)$ as in \eqref{tau1} and \eqref{bbb}. 
 More explicitly,  for any $g$ absolutely continuous with respect to $g_\infty(y)= C e^{-W(y)}$ with $W$ as in \eqref{pot} and \eqref{dri}, we have obtained
 \be\label{LS-gammagen}
  \int_{\R_+} g(y) \log \frac{g(y)}{g_\infty(y)}\, dy \leq \frac 1{2\rho(\delta)} \int_{\R_+} \left(
\frac{d}{d y} \log  \frac {g(y)}{g_\infty(y)}\right )^2 g(y)\, dy.
 \ee
 
Hence, one can conclude that the solution to the Fokker--Planck equation \fer{FP-z} converges exponentially fast towards the equilibrium density \fer{gg} in relative entropy, at the explicit rate $2\rho(\delta)$.
 
 \begin{rem}
The original Bakry--\'Emery theorem was restricted to potentials $W$ satisfying $W\in L^\infty ([0,+\infty))\cap C^2([0,+\infty))$, and it is not directly applicable in our case,  due to the behavior of $W$ near $0$ and near $+\infty$. However, inequality \fer{LS}  can be still proven when $W$ is derived from \fer{dri}, by resorting to an approximation argument, similar to the one considered for example in Ref. \cite{FPTT18}.
 \end{rem}
 
\begin{rem}\label{gGamma}
{\rm The constant $\rho$ in \fer{bbb} can be alternatively written in terms of the parameters $\kappa >2, \theta$ and $\nu> 2$ characterizing the generalized Gamma density \fer{gg}. In this case we obtain
 \be\label{new-c}
 \rho = \rho(\kappa,\theta, \nu) = \frac 1{\theta^2} (\kappa -1)^{1-2/\nu}\frac\nu{\nu -2}\left( \frac{\nu(\nu-1)(\nu-2)}2 \right)^{2/\nu}.
 \ee
 }
\end{rem}

%%%%%%%%%%%%%%%%%%%%%%%%%%%%%%%%%%%%%%%%%%%%%%%%%%%%

\subsection{Logarithmic Sobolev inequality with weight and exponential decay of relative entropy}

The results of Section \ref{deca} can be easily translated to the original Fokker--Planck equation \fer{FP2}. 

%In particular, the logarithmic Sobolev inequality \fer{LS} can be rewritten in terms of the Shannon entropy of a probability density relative to the inverse Gamma function \fer{new-eq}. 

 In this way, we obtain a new logarithmic Sobolev type inequality satisfied by the inverse Gamma densities $f_\infty^\delta$. To this end, let us introduce the
 weighted Fisher information $I_\delta$ of a probability density  $f$ relative to $f_\infty^\delta$ as follows
 \be\label{fis-peso}
 I_\delta(f|f_\infty^\delta) = \int_{\R_+}x^{2+\delta} \left(
\frac{d}{d x}  \log \frac{f(x)}{f_\infty^\delta(x)} \right )^2\, f(x) \, dx.
 \ee
We prove the following
  
 \begin{thm}\label{LS-theo}
 Let $\delta \in (0,1]$  and $f_\infty^\delta$ the inverse Gamma density defined in \eqref{new-equ}.
  For any
  probability density $f \in L^1(\R_+)$ absolutely continuous with respect to $f_\infty^\delta$ it holds
  \be\label{LS-peso}
 H(f,f_\infty^\delta) \leq \frac 1{2\rho(\delta)}  I_\delta(f, f_\infty^\delta).
 \ee
 The positive constant $\rho(\delta) >0$ is given in \eqref{tau1} and \eqref{bbb} .
 \end{thm}

 \begin{proof}
 We begin by considering the logarithmic Sobolev  inequality \eqref{LS-gammagen} satisfied by the generalized Gamma density.
By the change of  variable $y =  2/\delta \, x^{- \delta/ 2}$ as in \eqref{G},
inequality \eqref{LS-gammagen} becomes
\[
\begin{aligned}
&  \int_{\R_+} 
g(y(x)) \log 
\frac{g(y(x))\left| \frac{dy}{dx}\right|}{g_\infty(y(x))\left| \frac{dy}{dx}\right| }
\left| \frac{dy}{dx}\right| \, dx 
\\
  &\leq \frac 1{2\rho(\delta)}
   \int_{\R_+} 
   \left(
\frac{d}{d x} \log  \frac {g(y(x))\left| \frac{dy}{dx}\right| }{g_\infty(y(x))\left| \frac{dy}{dx}\right| } \left| \frac{dx}{dy}\right| 
\right )^2 g(y(x))\left| \frac{dy}{dx}\right| \, dx.
\end{aligned}
\]
Now, recalling the relation \eqref{equiv} between $g_\infty$ and $f_\infty^\delta$ and observing that $g$ is absolutely continuous with respect to $g_\infty$ if and only if $f(x) = g(y(x))\left| \frac{dy}{dx}\right|$ is absolutely continuous with respect to $f_\infty^\delta$, we get for all $f$  absolutely continuous with respect to $f_\infty^\delta$
\[
 \int_{\R_+} 
f(x) \log 
\frac{f(x)}{f_\infty^\delta(x)}
 \, dx \leq \frac 1{2\rho(\delta)}
   \int_{\R_+} x^{2+\delta}
   \left( \frac{d}{d x} \log  \frac {f(x) }{f_\infty^\delta(x)} 
\right )^2 f(x)\, dx.
\]
\end{proof}

It is still true \cite{FPTT17} that if $f(t)$ is a solution of the Fokker--Planck equation \eqref{FP2}, then 
at any given time $t >0$, the relative weighted Fisher information $I_\delta(f(t)|f_\infty^\delta)$ coincides with the entropy production at time $t$,  namely
 \[
 \frac d{dt}H(f(t)|f_\infty^\delta) = -I_\delta(f(t)|f_\infty^\delta).
 \]

The absolute continuity of the solution $f(t)$ with respect to the steady state $f_\infty^\delta$ (see Remark \eqref{ac}) and inequality \fer{LS-peso} then imply  exponential convergence of the solution to the Fokker--Planck equation \fer{FP2} towards the equilibrium density \fer{new-eq} at the explicit rate $2\rho(\delta)$. The result can be summarized in the following

\begin{thm} \label{expon}
Let $0<\delta \le 1$, and let $f(t)$ be the (unique) solution to the Fokker--Planck equation \fer{FP2}, corresponding to an initial value $f_0$ satisfying the conditions of Theorem \ref{Feller}.
Then, $f(t)$ converges exponentially fast in time towards equilibrium in relative entropy, and 
 \be\label{con-exp}
 H(f(t)|f_\infty^\delta)\le H(f_0|f_\infty^\delta) e^{-2\rho(\delta) \,t}.
 \ee
In particular, the Csiszar-Kullback inequality \cite{Csi,Kul}, implies exponential convergence in the $L_1(\R_+)$ sense at the explicit rate $\rho(\delta)$.
 
\end{thm}

\begin{rem}
If the initial value $f_0$ satisfies the assumptions of Theorem \ref{Feller}, then the ratio $ {f_0}/{f_\infty^\delta}$ is nonnegative and bounded. Hence the Shannon entropy  of $f_0$ relative to the equilibrium density $f_\infty^\delta$
is bounded.
\end{rem}

%\begin{rem}
% Thanks to the density in $L^1(\R_+)$ of the Feller's domain $\Sigma$  of initial values and to the linearity of the Fokker--Planck equation \eqref{FP2}, it is possible to conclude that for all $f_0$ probability density on $\R_+$, the corresponding solution of the Cauchy problem \eqref{bp-FP2} (which exists as it was pointed out in remark \eqref{desire}) converges exponentially fast in $L^1(\R_+)$ to the steady state. The lack of the physical no--flux boundary conditions  for the general $L^1$ solutions makes at this point this result merely theoretical.
%\end{rem}

\begin{rem}
{\rm As $\delta \to 0$, the positive constant $\rho(\delta)$ vanishes, and the uniform convexity of the potential $W(y)$ is lost. This unpleasant result can be obtained looking directly to the expression of the potential \fer{St} defined in Remark \ref{caso-no}. Indeed, in this case it holds
 \[
 W''(y) =\mu m e^{-y} \ge 0.
 \]
 Hence, the original Fokker--Planck equation introduced in Ref. \cite{BM} by Bouchaud and M\'ezard, and subsequently considered in Ref. \cite{CoPaTo05} as the grazing limit of a kinetic model of Boltzmann type, is equivalent to a Fokker--Planck equation with constant coefficient of diffusion and drift given by a potential that it is not uniformly convex. Hence, for this Fokker--Planck equation exponential convergence towards equilibrium in relative entropy does not follows directly from a logarithmic Sobolev inequality.
 }
\end{rem}

%%%%%%%%%%%%%%%%%%%%%%%%%%%%%%%%%%%%%%%%%%%%%%%
\subsection{Exponential decay for general $L^1$-data}

As pointed out in Remark \ref{desire},  in Ref. \cite{Fel2} Feller proved that for any probability density $f_0 \in L^1(\R_+)$,  the Cauchy problem \eqref{bp-FP2}, posed in  $C([0,+\infty), L^1(\R_+))  \cap C^1([0,+\infty), L^1(\R_+))$  has a  solution defined through the semigroup generated by the operator $\Omega^\star$ recalled in the proof of Theorem \ref{Feller}. Moreover, this semigroup is a contraction in $L^1(\R_+)$. Therefore, given two initial densities $f_0$ and $g_0$,  for any time $t >0$  the solutions $f(t)$ and $g(t)$ to the the Cauchy problem \eqref{bp-FP2} with initial data  $f_0$  and $g_0$ respectively satisfy
 \be\label{contr}
 \| f(t) -g(t)\|_{L^1(\R_+)} \le  \| f_0 -g_0\|_{L^1(\R_+)}.
 \ee
 If the solution $g(t)$ converges exponentially fast to equilibrium, and $f_0$ and $g_0$ are close enough, it is reasonable to guess that the solution $f(t)$ does the same. This is indeed what we are going to prove, provided that the entropy of the initial density $f_0$ relative to the equilibrium  is assumed to be bounded. 
 \vskip 1cm
   \begin{thm} \label{expon2}
Let $0<\delta \le 1$, and let $f(t)$ be the solution to the Fokker--Planck equation \fer{FP2} corresponding to a probability density $f_0\in L^1(\R_+)$ defined through the semigroup generated by $\Omega^*$.  Suppose moreover that the density $f_0$ satisfies the following conditions
\begin{align}
&\int_1^{+\infty} f_0(w) w^\alpha d w<+\infty \label{mom}\\
&\int_0^1 \frac {f_0(w)} w dw <+\infty \label{int-zero}\\
&\int_{\R_+} f_0(w) \log_+f_0(w) dw <+\infty \label{entr}
\end{align}
for a  constant $\alpha \in (0,1] $.
Then, $f(t)$ converges exponentially fast in time towards equilibrium in $L^1(\R_+)$, and there exists a positive constant $K= K(f_0)$ such that 
 \be\label{con-exp2}
  \| f(t) - f_\infty^\delta\|_{L^1(\R_+)} \le 2 K^{1/2}e^{-\rho(\delta) \,t}
 \ee
 where $\rho(\delta)$ is the positive constant rate appearing in formulas \eqref{tau1} and \eqref{bbb}.
\end{thm}    

\begin{rem}\label{entropia}
Conditions \eqref{mom} and \eqref{entr} imply
$ f_0 \log f_0\in L^1(\R_+)$.
Indeed, let  $0<\e < \frac \alpha{\alpha +1} <1$. Then for some positive constant $C=C(\epsilon)$ we get
\[
\begin{aligned}
&\int_{\R_+} f_0(w) \log_- f_0(w) dw = \int_{\R_+} f_0(w) (-\log f_0(w)) \chi_{\{f_0(w) \leq 1\}} dw \\
&\leq C \int_{\R_+} f_0(w) \frac 1{(f_0(w))^\e} \chi_{\{f_0(w) \leq 1\}} dw \\
& =C \int_0^1(f_0(w))^{1-\e} \chi_{\{f_0(w) \leq 1\}} dw +
C\int_1^{+\infty} (f_0(w))^{1-\e}  \chi_{\{f_0(w) \leq 1\}} dw\\
&\leq C + C\int_1^{+\infty} \frac 1{w^{\alpha (1-\e)}}(f_0(w))^{1-\e} w^{ \alpha (1-\e)} dw \\
&\leq C+ C\left(\int_1^{+\infty} \frac 1 {w^ \frac {\alpha(1-\e)}\e}dw \right )^\e  \left(\int_1^{+\infty}  f_0(w) \ w^\alpha\  dw \right )^{1-\e }<+\infty. 
\end{aligned}
\]
In the last line we applied H\"older inequality.  The integrals are bounded since the exponent   $ \frac {\alpha(1-\e)}\e >1$,  and  \fer{mom} holds true. 
\end{rem}

\begin{rem}
The assumptions made on $f_0$ imply the boundedness of the  relative entropy $H(f_0| f_\infty^\delta)$. 
Since
\[
\int_{\R_+} f_0(w) \log \frac  {f_0(w)}{f_\infty^\delta(w)}  dw =\int_{\R_+} f_0(w) \log f_0(w) dw -
\int_{\R_+} f_0(w) \log f_\infty^\delta(w)  dw,
\]
and the Shannon entropy of the initial value is bounded, the boundedness of the relative entropy follows by the boundedness of the second term. One has
\begin{equations}\label{sum2}
&\int_{\R_+} f_0(w) \log f_\infty^\delta(w)  dw =  \int_{\R_+} f_0(w) \log \left(  \frac{(\mu\, m)^{1+\delta + \mu}}{\Gamma(1+\delta + \mu)}\frac{\exp\left(-\frac{\mu\, m}{w}\right)}{w^{2+\delta +\mu}}
\right )  dw\\
&=  \log \left(  \frac{(\mu\, m)^{1+\delta + \mu}}{\Gamma(1+\delta + \mu)}\right)\int_{\R_+} f_0(w) \, dw  - \\ &(2+\delta + \mu) \int_{\R_+}  f_0(w)\log w \, dw  - \mu m\int_{\R_+} {f_0(w)}\frac 1w \, dw
 \\
&= C_1 \int_{\R_+} f_0(w) dw  +  C_2\int_{\R_+} {f_0(w)}\log w\, dw +  C_3\int_{\R_+}  {f_0(w)}\frac 1w \, dw,
\end{equations}
for some real constants $C_1$, $C_2$  and $C_3$.
Therefore
\begin{equation}
\left|\int_{\R_+} f_0(w) \log f_\infty^\delta(w)  dw \right|\le D_1 \int_{\R_+} f_0(w) dw  + D_2\int_{\R_+} w^\alpha {f_0(w)} \, dw +D_3\int_{\R_+} \frac {f_0(w)} w dw,
\end{equation}
for some positive constants  $D_1$, $D_2$  and $D_3$.
 Indeed, 
 \[
\left| \int_{\R_+}  f_0(w)  \log w \, dw \right| \le \int_{0}^1 \frac {f_0(w)} w dw + \frac 1\alpha \int_1^{+\infty} w^\alpha {f_0(w)} \, dw.
 \]
 
This shows a deep analogy with the assumptions  which imply convergence to equilibrium for the classical Boltzmann equation \cite{TV,Vi1} In that case, in dimension $n \ge 1$, the equilibrium is the Mawellian density
$M(w)=\frac 1{(2\pi )^{n/2}} e^{-\frac {|w|^2}{2}}$, and the assumption
 \[
 \int_{\R^n} (1+|w|^2 + |\log f_0(w)|) f_0(w) dw = C < +\infty
 \]
 guarantees both the boundedness of the entropy of the initial datum $f_0$ relative to the Maxwellian 
 \[
\int_{\R^n} f_0(w) \log \frac   {f_0(w)}{M(w)}  dw <+\infty,
 \]
 and the convergence to equilibrium in $L^1(\R^n)$.
\end{rem}

\begin{proof}
 Let us consider an initial density $f_0$ satisfying assumptions \eqref{mom}, \eqref{int-zero} and \eqref{entr}.
 Since the space $\Sigma$ is dense in $L^1(\R_+)$, for any given $\epsilon \ll 1$ we can find a probability density  $g_\epsilon \in \Sigma$ such that its relative entropy with respect to the steady state
$f_\infty^\delta$ (which is  a positive quantity due to the bound $x\log x-x+1\geq 0$ for any $x>0$)  is bounded by a constant depending only on $f_0$
 \be\label{hh}
  H(g_\epsilon |f_\infty^\delta) \le  K(f_0),
   \ee
  and at the same time
  \be\label{jj}
   \| f_0 -g_\epsilon\|_{L^1(\R_+)} \le \epsilon. 
   \ee 
 To this aim, for a given $\epsilon >0$,  let us define
   \[
   \tilde f_\e(w):= \frac {f_0(w) \chi_{\{ \e <w <\frac 1\e\}}}{\| f_0\chi_{\{ \e <w <\frac 1\e\}}\|_{L^1(\R_+)}}.
   \]
  Then, $\tilde f_\e$ is a probability density for all $\e>0$.
  Further, let us introduce a function $\phi \in C^\infty (\R_+)$ satisfying $ 0\leq \phi (w)\leq 1$ for all $w\in \R_+$, $\int_{\R_+} \phi (w)dw =1$, $\mathrm {supp }\ \phi \subset \left[ \frac 12, 2\right ]$, $\phi(w)= 1$ on $\left[ 1, \frac 32\right ]$. Let 
   \[
   \phi_\e(w)= \frac 1\e \phi\left(\frac w\e\right ).
   \]
 denote a dilation of $\Phi$, which is such that  $\int_{\R_+} \phi_\e (w)dw =1$ and $\mathrm {supp }\ \phi_\e \subset \left[ \frac 12 \e, 2\e\right ]$  for all $\e >0$.
   Finally, let us introduce the function
   \be\label{appr}
   g_\eps (w):= \e f_\infty^\delta (w)+ (1-\e) \tilde f_\e \ast \phi_\e (w).
   \ee
  In definition \fer{appr} $*$ denotes the convolution operation in $\R_+$, classically defined on two functions $f$ and $g$ as
   \be\label{conv}
   f\ast g (w) = \int_0^w f(w-v)g(v) dv.
   \ee
   It is straightforward to verify that $g_\e$ is a probability density for all $\e>0$.
 By construction,  we can easily prove that $g_\e \in \Sigma$. Indeed, since the convolution term is smooth, the function $g_\e$ has the same regularity as the steady state $f_\infty^\delta$.  
Let us now estimate  the support of the convolution term 
 \[
   \tilde f_\e \ast \phi_\e (w) =\int _0^w \tilde f_\e(w-v) \phi_e (v) dv.
   \]
 Since $\mathrm {supp }\ \phi_\e \subset \left[ \frac 12 \e, 2\e\right ]$ and $\mathrm {supp }\ \tilde f_\e \subset \left[\e, \frac 1 \e\right ]$ then
 $\mathrm {supp }\ \tilde f_\e \ast \phi_\e \subset \left [ \e +\frac 12 \e, \frac 1\e +2\e\right ]$.
In particular, $g_\e $ coincides with $\e f_\infty^\delta$ in a neighborhood of $0^+$ and of $+\infty$ and satisfies the no--flux boundary conditions.

In addition, $g_\e$ converges to $f_0$ in $L^1(\R_+)$ for $\e \to 0$. 
In fact
\[
\begin{aligned}
\|g_\eps -f_0\|_{L^1(\R_+)} & =  \| \e f_\infty^\delta+ (1-\e) \tilde f_\e \ast \phi_\e-f_0\|_{L^1(\R_+)}\\
& \leq \e \|f_\infty^\delta\|_{L^1(\R_+)} + \| (1-\e) \tilde f_\e \ast \phi_\e-f_0\|_{L^1(\R_+)}.
\end{aligned}
\]
The first term of course vanishes for $\eps \to 0$.
The second term satisfies the inequality
\[
\| (1-\e) \tilde f_\e \ast \phi_\e-f_0\|_{L^1(\R_+)} \leq \| (1-\e)( \tilde f_\e \ast \phi_\e-f_0 )\|_{L^1(\R_+)} + \e \|f_0\|_{L^1(\R_+)}.
\]
The last term  vanishes for $\e \to 0$. The remaining term  satisfies
\begin{equations}\label{term}
& (1-\e)  \| \tilde f_\e \ast \phi_\e-f_0\|_{L^1(\R_+)}\leq  \\ &(1-\e) \| ( \tilde f_\e - f_0 ) \ast \phi_\e \|_{L^1(\R_+)} + (1-\e) \| f_0\ast \phi_\e -f_0\|_{L^1(\R_+)}.
 \end{equations}
 The last term in \fer{term} vanishes in reason of the fact that $\phi_\e$ is an approximation of the identity in $L^1(\R_+)$ (see e.g. Ref. \cite{dan}, Lemma 3 p. 481).
 The first term $( \tilde f_\e - f_0 ) \ast \phi_\e$ in \fer{term}  converges to 0, as $\e$ tends to 0, thanks to Fubini and Lebesgue dominated convergence Theorems.
% \[
 %\begin{aligned}
%&\| ( \tilde f_\e - f_0) \ast \phi_\e \|_{L^1(\R_+)} = \int_0^{+\infty} \left | \int_0^w \left (\tilde f_\e(w-v)-f_0(w-v)\right ) \phi_\e(v) dv\right | dw \\
%&\leq  \int_0^{+\infty} \int_0^w \left |\tilde f_\e(w-v)-f_0(w-v)\right | \phi_\e(v) dv  dw \\
%&=   \int_0^{+\infty} \int_v^{+\infty} \left |\tilde f_\e(w-v)-f_0(w-v)\right | \phi_\e(v) dw  dv\\
%&=   \int_0^{+\infty} \left(\int_0^{+\infty} \left |\tilde f_\e(s)-f_0(s)\right | ds\right )  \phi_\e(v) dv  = \|\tilde f_\e -f_0\|_ {L^1(\R_+)} \\
%&= \int_0^{+\infty} \left|   \frac {f_0(w)\chi_{\{\eps<w<\frac 1\e\}}} {\|f_0\chi_{\{\eps<w<\frac 1\e\}}\|_{L^1}} -f_0(w)\right |dw\\
%&\leq \frac 1 {\|f_0\chi_{\{\eps<w<\frac 1\e\}}\|_{L^1}} \int_0^{+\infty} |f_0(w)\chi_{\{\eps<w<\frac 1\e\}} -f_0(w)| dw \\
%&\qquad\qquad + \left(\frac 1 {\|f_0\chi_{\{\eps<w<\frac 1\e\}}\|_{L^1}}  -1\right ) \|f_0\|_{L^1}
%\to 0, \quad \e \to 0
%\end{aligned}
%\]
Hence, we can conclude that $g_\e \to f_0$ in $L^1(\R_+)$ as $\e \to 0$.

Let us now prove  that we can find a positive constant $K(f_0) >0$  such that, for  small $\e >0$
\be\label{unif-bound}
H(g_\e|f_\infty^\delta) \leq K (f_0).
\ee
(In the rest of the proof the constant $K$ could vary from one line to another).
To this extent, let us estimate separately the two terms 
\be\label{hg}
H(g_\e)= \int_{\R_+} g_\e(w) \log g_\e(w) dw
\ee
(the Shannon entropy of $g_\e$) and 
\be\label{int}
 A(g_\e)= \int_{\R_+} g_\e(w) \log f_\infty^\delta(w) dw.
 \ee
If both are bounded, then we have $H(g_\e|f_\infty^\delta) = H(g_\e) - A(g_\e)$.
Let us begin by estimating $H(g_\e)$. We remark that $H(g_\e)$ is not necessarily positive. 
Since $g_\e \to f_0$ in $L^1$, by the lower  semi-continuity of the Shannon entropy, we obtain, for all $\e$ small enough, the lower bound 
\be\label{sotto}
H(g_\e) \geq   H(f_0)-1.
\ee
Moreover, since $g_\e$ is a convex combination of two terms and the Shannon entropy $H$ is a convex functional, then 
\[
H(g_\e) \leq \e H(f_\infty^\delta) + (1-\e) H(\tilde f_\e\ast \phi_\e).
\]
Further, proceeding as in Ref. \cite{LT} we get the bound
\[
H(\tilde f_\e\ast \phi_\e) \leq H(\tilde f_\e) =\frac{1}{\|f_0\chi_{\{\e <w<\frac 1\e\}}\|_{L^1}}\int_{\R_+} f_0(w) \log f_0(w) \chi_{\{\e <w<\frac 1\e\}} dw.
\]   
Since $f_0 \log f_0 \in L^1(\R_+)$ (see Remark \ref{entropia}), we conclude by the Lebesgue dominated convergence   that $H(\tilde f_\e) \to H(f_0)$ for $\e \to 0$. So, we get for small $\e$
\begin{equation*}
H(g_\e) \leq H(f_0)+1
\end{equation*}
and  exploiting  \eqref{sotto} we can obtain 
\be\label{sopra}
|H(g_\e)| \leq  |H(f_0)|+1.
\ee
Let us come to the term 
\be\label{K}
A(g_\e) =  \e \int_{\R_+} f_\infty^\delta (w) \log  f_\infty^\delta (w) dw +   (1-\e) \int_{\R_+} \tilde f_\e \ast \phi_\e (w) \log  f_\infty^\delta (w) dw.
\ee
The first term is  bounded.
For the second one, we obtain as in \fer{sum2} 
\begin{align}
 & \int_{\R_+} \tilde f_\e \ast \phi_\e (w) \log  f_\infty^\delta (w) dw\notag \\
 & = \int_{\R_+} \tilde f_\e \ast \phi_\e (w) \log  \left(  \frac{(\mu\, m)^{1+\delta + \mu}}{\Gamma(1+\delta + \mu)}\frac{\exp\left(-\frac{\mu\, m}{w}\right)}{w^{2+\delta +\mu}} \right )
  dw\notag\\
  &= C_1 \int_{\R_+}  \tilde f_\e \ast \phi_\e (w)\  dw + C_2  \int_{\R_+}  \tilde f_\e \ast \phi_\e (w) \log w\ dw + C_3   \int_{\R_+} \tilde f_\e \ast \phi_\e (w)\ \frac 1 w\ dw\label{treterm}
    \end{align}
for $C_1$, $C_2$ and $C_3$ explicit constants in $\R$.
Since $ \tilde f_\e \ast \phi_\e$ is a probability density,
\be\label{l1}
\| \tilde f_\e \ast \phi_\e\|_{L^1(\R_+)}=1.
\ee
Then we split the second integral into
\[
\begin{aligned}
  \int_{\R_+}  \tilde f_\e \ast \phi_\e (w) \log w\ dw &=   \int_{0}^1  \tilde f_\e \ast \phi_\e (w) \log w\ dw +   \int_1^{+\infty}  \tilde f_\e \ast \phi_\e (w) \log w \ dw\\
  &= A_1+A_2.
  \end{aligned}
  \]
 We have
 \be\label{b1}
 0\leq -A_1 \leq  \int_{0}^1  \tilde f_\e \ast \phi_\e (w) \ \frac 1 w\ dw = B_1
 \ee
and
\be\label{b2}
0\leq A_2 \leq \int_1^{+\infty}  \tilde f_\e \ast \phi_\e (w) \ w^\alpha\ dw =B_2.
\ee
Now, since $0\leq v\leq w$ implies $\frac 1 w \leq \frac 1v$ we get
\[
\begin{aligned}
B_1 &= \int_0^1 \int_0^w \tilde f_\e (v) \phi_\e (w-v) \frac 1w dv dw\leq  \int_0^1 \int_0^w \frac { \tilde f_\e (v)} v \phi_\e (w-v) dv dw\\
&\leq  \left\|\frac { \tilde f_\e (w )} w\ast \phi_\e (w) \right \|_{L^1(\R_+)}.
\end{aligned}
\] 
Since by assumption \eqref{int-zero}  $ {\tilde f_\e (w )}/w \to  {f_0 (w )}/ w$ in $L^1(\R_+)$,  proceeding as in estimate \eqref{term} we  prove that $({ \tilde f_\e (w)}/ w) \ast \phi_\e (w) \to  { f_0 (w )}/ w$ in $L^1(\R_+)$. Hence there is a positive constant $K$ such that 
\be\label{bound-B1}
B_1 \leq K \left\| \frac {f_0(w)}{w}\right \|_{L^1(\R_+)}.
\ee
Let us now evaluate the $B_2$ term. By the classical inequality $(x+y)^\alpha \leq x^\alpha +y^\alpha$, which holds for positive $x$, $y$ and $0<\alpha \leq1$, we get
\[
\begin{aligned}
B_2 &= \int_1^{+\infty}  \int_0^w \tilde f_\e(w-v) \phi_\e (v)  w^\alpha\ dv dw\\
&\leq  \int_1^{+\infty}  \int_0^w (w-v)^\alpha \tilde f_\e(w-v) \phi_\e (v) \ dv dw + \int_1^{+\infty}  \int_0^w  \tilde f_\e(w-v) v^\alpha \phi_\e (v) \ dv dw\\
&\leq  \int_1^{+\infty}  \left(w^\alpha  \tilde f_\e(w)\right ) \ast \phi_\e (w) \ dw  + \e^\alpha  \int_1^{+\infty}    \tilde f_\e(w) \ast \psi_\e(w) \ dw \\
&\leq  \left\|  \left(w^\alpha \tilde f_\e (w )\right ) \ast \phi_\e (w) \right \|_{L^1(\R_+)} + \e^\alpha \left\|  \tilde f_\e(w) \ast \psi_\e(w) \right \|_{L^1(\R_+)}
\end{aligned}
\] 
where 
 \[
\psi_\e(w)= \frac 1\e \left(\frac w\e\right )^\alpha \phi(\frac w\e).
\]
By assumption \eqref{mom} we have $w^\alpha \tilde f_\e (w ) \to w^\alpha f_0 (w )$ in $L^1(\R_+)$, 
therefore $ \left(w^\alpha \tilde f_\e (w )\right ) \ast \phi_\e (w)$ converges to $ w^\alpha f_0 (w )$ in $L^1(\R_+)$. Thus  we conclude that there is $K>0$ such that 
\be\label{bound-B2-1}
\left\|  \left(w^\alpha \tilde f_\e (w )\right ) \ast \phi_\e (w) \right \|_{L^1(\R_+)} \leq K \left\| w^\alpha f_0\right \|_{L^1(\R_+)}.
\ee
Moreover
\be\label{bound-B2-2}
\e^\alpha \left\|  \tilde f_\e(w) \ast \psi_\e(w) \right \|_{L^1(\R_+)} \leq  \e^\alpha \|f_0\|_{L^1(\R_+)} \| \psi \|_{L^1(\R_+)}\to 0, \quad \e \to 0.
\ee
Hence by \eqref{bound-B2-1} and \eqref{bound-B2-2}, there exists $K>0$ such that  for small $\e$
\be\label{bound-B2}
B_2 \leq K  \left\| w^\alpha f_0\right \|_{L^1(\R_+)}.
\ee
In view of  \eqref{b1}, \eqref{b2}, \eqref{bound-B1} and \eqref{bound-B2} we finally get
\be\label{log}
\left| \int_{\R_+}  \tilde f_\e \ast \phi_\e (w) \log w\ dw \right | \leq K \left(  \left\| \frac {f_0(w)}{w}\right \|_{L^1(\R_+)} +    \left\| w^\alpha f_0\right \|_{L^1(\R_+)}  \right ).
\ee
The last integral term in \eqref{treterm} can be bounded as follows
\[
\begin{aligned}
\int_{\R_+} \tilde f_\e \ast \phi_\e (w)\ \frac 1 w\ dw &= \int_0^1 \tilde f_\e \ast \phi_\e (w)\ \frac 1 w\ dw + \int_1^{+\infty} \tilde f_\e \ast \phi_\e (w) \ \frac 1w\ dw\\
&\leq B_1 + \| \tilde f_\e \ast \phi_\e \|_{L^1(\R_+)}.
\end{aligned}
\]
Hence, by  \eqref{bound-B1}  there exists $K>0$ such that 
\be\label{1sux}
\int_{\R_+} \tilde f_\e \ast \phi_\e (w)\ \frac 1 w\ dw \leq K  \left(  \left\| \frac {f_0(w)}{w}\right \|_{L^1(\R_+)} +    \left\| f_0\right \|_{L^1(\R_+)}  \right ).
\ee
Collecting inequalities \eqref{l1}, \eqref{log} and \eqref{1sux} we obtain that the term $A(g_\e)$ defined in \eqref{int}
can be bounded as follows
\[
|A(g_e)| \leq K \left(|H(f_\infty^\delta)| + \|f_0\|_{L^1(\R_+)} +  \left\| \frac {f_0(w)}{w}\right \|_{L^1(\R_+)} +    \left\| w^\alpha f_0\right \|_{L^1(\R_+)}\right ).
\]
Recalling inequality \eqref{sopra}, we have
 that, for some $K>0$ and for all $\eps>0$ small enough,
\[
\begin{aligned}
H(g_\e|f_\infty^\delta) &\leq |H(g_\eps)|+ |A(g_\e)| \\
&\leq K  \left(|H(f_0)|+1+ |H(f_\infty^\delta)| + \|f_0\|_{L^1(\R_+)} +  \left\| \frac {f_0(w)}{w}\right \|_{L^1(\R_+)} +    \left\| w^\alpha f_0\right \|_{L^1(\R_+)}\right )\\
&:= K(f_0).
\end{aligned}
\]
 For any given time $t>0$, let us fix
   \be\label{ce}
   \epsilon \le K(f_0)^{1/2}e^{-\rho(\delta) \,t}.
   \ee
  By means of inequality \fer{contr}, 
    \begin{equations}\label{xx}
    \| f(t) - f_\infty^\delta\|_{L^1(\R_+)} &\le \| f(t) -g(t)\|_{L^1(\R_+)}  + \| g(t) - f_\infty^\delta\|_{L^1(\R_+)} \\
    &\le\| f_0 -g_\eps\|_{L^1(\R_+)} + \| g(t) - f_\infty^\delta\|_{L^1(\R_+)}.
    \end{equations}
 Since $g(t)$  is solution to the Fokker--Planck equation \fer{FP2}, corresponding to an initial value $g_\eps$ satisfying the conditions of Theorem \ref{Feller}, thanks to Theorem \ref{expon} and condition \fer{hh}
  \[
 H(g(t)|f_\infty^\delta)\le H(g_\eps|f_\infty^\delta) e^{-2\rho(\delta) \,t} \le K(f_0) e^{-2\rho(\delta) \,t}.
 \]
Hence, by Csiszar--Kullbach inequality \cite{Csi,Kul}, we obtain
 \be\label{cs}
 \| g(t) - f_\infty^\delta\|_{L^1(\R_+)} \le  K(f_0) ^{1/2}e^{-\rho(\delta) \,t},
 \ee
 and \fer{xx} implies
  \be\label{zz}
    \| f(t) - f_\infty^\delta\|_{L^1(\R_+)} \le \epsilon +  K(f_0)^{1/2}e^{-\rho(\delta) \,t}.
      \ee
 Therefore, thanks to condition \fer{ce}, for any given $t >0$ it holds
  \be\label{z2}
    \| f(t) - f_\infty^\delta\|_{L^1(\R_+)} \le 2  K(f_0)^{1/2}e^{-\rho(\delta) \,t}.
      \ee
 This concludes the proof.
\end{proof}
      
 \begin{rem}
 Note that, unlike the result of Theorem \ref{expon}, we do not know if the solution to the Fokker--Planck equation with initial density in $L^1(\R_+)$ is exponentially convergent towards equilibrium in relative entropy. However, it is enough to choose an initial density close to equilibrium in the sense of relative entropy, to have exponential convergence in $L^1(\R_+)$ at explicit rate. 
 \end{rem}     
%%%%%%%%%%%%%%%%%%%%%%%%%%%%%%%%%%%%%%%%%%%%%%%%%%%%

\section{Nonlinear models}\label{n-l}

The modeling assumptions of Section \ref{model2}, leading to the Boltzmann-type equation \fer{line} with the non-Maxwellian kernel \fer{Ker}, and subsequently, via the grazing limit, to the linear Fokker--Planck equation \fer{FP2}, can be easily extended to cover binary trading.
Binary trading between agents, with saving propensity and risk, has been considered in Ref. \cite{CoPaTo05}. Similarly to \fer{lin}, when two agents with wealths $v$ and $w$ interact, the post trade wealths change into
\begin{equation}
  \label{eq.cpt}
  v^* = \Bigl(1-\lambda+\eta_1\Bigr)v + \lambda w, \quad
  w^* = \Bigl(1-\lambda +\eta_2\Bigr)w + \lambda v,
\end{equation}
where $0 <\lambda <1$ is the parameter which identifies the saving propensity of agents.
The coefficients $\eta_1,\eta_2$ are random parameters, which are
independent of $v$ and $w$, and distributed so that always
$v^*,\,w^*\geq  0$, i.e.\ $\eta_1,\,\eta_2\geq \lambda -1$.  Unless these
random variables are centered, i.e.\
$\langle\eta_1\rangle=\langle\eta_2\rangle=0$, it is immediately
seen that the mean wealth is not preserved, but it increases or
decreases exponentially (see the computations in Ref.~ \cite{CoPaTo05}).
For centered $\eta_i$,
\be\label{mean-c}
  \langle v^* + w^* \rangle = (1+\langle\eta_1\rangle) v
  + (1+\langle\eta_2\rangle) w = v + w ,
\ee
implying conservation of the average wealth. If we introduce the kernel \fer{Ker} as in equation \fer{line}, the wealth density satisfies the bilinear kinetic model 
 \begin{equation}
  \label{biline}
 \frac{d}{dt}\int_{\R_+}f(t,w)\varphi(w)\,dw  = 
 \kappa\, \Big \langle \int_{\R_+\times \R_+} \left( vw\right)^\delta \left( \varphi(w^*)-\varphi(w) \right)f(t,v) f(t,w)
\,dv\,dw\Big \rangle.
 \end{equation} 
The model includes the standard Maxwellian  model considered in Ref. \cite{CoPaTo05}, which is obtained for $\delta = 0$.

Note that, in consequence of \fer{mean-c}, the solution to equation \fer{biline} is such that both mass and mean wealth are preserved in time. Therefore, if the initial value is a probability density of mean value $m$, we get for all $t>0$
 \be\label{con2}
 \int_{\R_+} f(t,w) \, dw = 1, \quad \int_{\R_+}w\, f(t,w) \, dw = m.
 \ee 
As discussed in Section \ref{gambling}, the steady state $f_\infty$ of the kinetic model \fer{biline} is related to the steady state $g_\infty$ of the Maxwellian model, that  solves equation \fer{equ} with $w^*$ now given as in \fer{eq.cpt}, by the relation
 \[
 g_\infty(w) = w^\delta f_\infty(w).
 \]
In this case no explicit equilibria are available. Nevertheless,  the analysis of Ref.~ \cite{MaTo07}  shows that the microscopic interaction \fer{eq.cpt}  is such that the steady state of the Maxwellian model is able to describe all interesting behaviors of wealth distribution and the results relative to $g_\infty$ easily translate to $f_\infty$.  

In more details, precise results have been obtained in Ref.~ \cite{MaTo07} if the random variables $\eta_i$, $i =1,2$ assume only two values, that is $\eta_i=\pm r$, where
each sign occurs with probability $1/2$. The factor
$r\in(0,\lambda)$ quantifies the
risk of the market.  Within this choice, the numerical evaluation shows that the increasing of the risk parameter determines the tails of the equilibrium density, passing from \emph{slim tails} for low values to \emph{Pareto tails} for large values. As in the pure gambling case, the relationship between equilibria implies that in the non-Maxwellian model the
Pareto index of the steady state increases of an exponent $\delta$ with respect to the Maxwellian one.

By following step-by-step the procedure of Section \ref{quasi}, one realizes that, by applying the scaling of time $t \to \e t$ and \fer{scal} the weak form of the kinetic model \fer{biline} is well approximated by the weak form of a nonlinear Fokker--Planck equation (with variable coefficients), given by
\begin{equations}
  \label{m-14}
 & \frac{d}{dt}\int_{\R_+}\varphi(w) \,f(t,w)\,dw = \\
  & \kappa \int_{\R_+} \left( \varphi'(w)\,\lambda\, w^\delta(m_{1+\delta}(t) - m_\delta(t) w) + \frac 12 \varphi''(w) \sigma\, m_\delta(t) w^{2+\delta} \right) f(t,w)\, dw. 
 \end{equations}
In equation \fer{m-14} $m_\alpha(t)$ defines the moment of order $\alpha >0$ of the solution. 
By choosing $\varphi(w) = 1, w$ in \fer{m-14}, one shows that both the mass density and the mean wealth are preserved in time, so that \fer{con2} hold.

If boundary conditions on $w=0$ and $w=+\infty$ are added,  such that the boundary terms produced by the integration by parts vanish, equation \fer{m-13} coincides again with the weak form of the (nonlinear) Fokker--Planck equation
 \begin{equation}\label{FP23}
 \frac{\partial f}{\partial t} = \kappa\, m_\delta(t) \frac\sigma 2 \frac{\partial^2 }{\partial w^2}
 \left(w^{2+\delta}f \right )+ \lambda\,
 \frac{\partial}{\partial w}\left( w^\delta(m_\delta(t) \, w - m_{1+\delta}(t)) f\right).
 \end{equation}
It is clear that the presence of unknown time-dependent coefficients makes the qualitative study of equation \fer{FP23} a challenging problem. In particular, it would be interesting to know if the large-time behavior of the solution to \fer{FP23} is exponentially convergent towards equilibrium.

%%%%%%%%%%%%%%%%%%%%%%%%%%%%%%%%%%%%%%%%%%

\section{Conclusions}

In this paper, we introduced and studied kinetic models of Boltzmann type, describing the evolution of wealth distribution in a multi-agent society, previously studied with the Maxwellian kernel approximation \cite{CoPaTo05,MaTo07,PT13} The main novelty of the present approach was to introduce a non-Maxwellian kernel in the collision integral, suitable to exclude economically irrelevant interactions. In the linear case, the resulting Fokker--Planck description possesses a steady state distribution of shape identical to that resulting from the Maxwellian one. 
However its solution, in contrast  with the Maxwellian description \cite{TT,TT18}, has been shown to converge exponentially fast in relative Shannon entropy towards equilibrium with an explicit rate for a  class of regular initial data, and to converge exponentially in $L^1(\R_+)$ at explicit rate for all initial densities that have bounded relative entropy with respect to the equilibrium density. 

It is worth noticing that the right large-time behavior of the wealth distribution depends upon an economic improvement of the original model, obtained in terms of a collision kernel in the Boltzmann equation. Resorting to an interaction kernel is clearly a possibility  easily included in a Boltzmann-type description. Indeed, in the classical picture of collisions between molecules, the collision kernel is an essential ingredient that classifies the type of interaction. 

We are confident that this idea could be used in a profitable way in other socio-economic applications of kinetic theory, in order to respect at best the agent's behavior and to obtain at the same time marked improvements of the underlying mathematical models.

 %%%%%%%%%%%%%%%%%%%%%%%%%%%%%%%%%%%%%%%%%%%%%%%%%%%%%%%%%%%%%%%%%%%%%%%%%%%%%%%

\section*{Acknowledgement} This work has been written within the
activities of GNFM and GNAMPA groups  of INdAM (National Institute of
High Mathematics), and partially supported by  MIUR project ``PRIN 2017TEXA3H - Gradient flows, Optimal Transport and Metric Measure Structures''. The research was partially supported by
the Italian Ministry of Education, University and Research (MIUR): Dipartimenti
di Eccellenza Program (2018--2022) - Dept. of Mathematics ``F.
Casorati'', University of Pavia.

%%%%%%%%%%%%%%%%%%%%%%%%%%%%%%%%%%%%%%%%%%%%%%%%%%%%%%%%%%%%%%%%%%%%%%%%%%%%%%%

\end{document}